\documentclass[10pt,twoside]{siamltex}
\usepackage{amsfonts,epsfig}
\usepackage{slashbox}
\usepackage{color}
\usepackage{amsmath,multirow}
\usepackage{amsmath}
\usepackage{mathrsfs}
\usepackage{amssymb}
\usepackage{booktabs}
\usepackage{url}

\usepackage[toc,page]{appendix}
\usepackage{algorithm}
\usepackage{algorithmicx}
\usepackage{algpseudocode}

\usepackage[font=scriptsize,labelsep=space]{caption}

\newtheorem{example}[theorem]{Example}
\newtheorem{remark}[subsection]{Remark}
\DeclareMathOperator{\tridiag}{tridiag}
\newcommand{\bs}{\boldsymbol}

\begin{document}

\bibliographystyle{plain}

\title{On circulant and skew-circulant splitting algorithms for (continuous) Sylvester equations
\thanks{
	The research of the last two authors was partially financed by Portuguese Funds
	through FCT (Fundação para a Ciência e a Tecnologia) within the Projects UIDB/00013/2020 and
	UIDP/00013/2020. 
This work was also supported by NSFC (National Natural Science Foundation of
China) under the Grant No. 11371075, the Hunan Key Laboratory of Mathematical Modeling and Analysis in Engineering.
}}

\author{
	Zhongyun\ Liu\thanks{School of Mathematics and Statistics,
		Changsha University of Science and Technology, Changsha 410076, P.
		R. China (liuzhongyun@263.net). }\and Fang\ Zhang\footnotemark[1]
\and Carla\ Ferreira\thanks{Centro de Matem\'atica,
	Universidade do Minho, 4710-057 Braga, Portugal (caferrei@math.uminho.pt, zhang@math.uminho.pt).}\and Yulin\ Zhang\footnotemark[3]
}
%
\pagestyle{myheadings} \markboth{Z. Y. \ Liu, F. Zhang, C. Ferreira \ and Y. L. Zhang}{Circulant and skew-circulant splitting algorithms}
\maketitle

\begin{abstract}
We present a circulant and skew-circulant splitting (CSCS) iterative method for solving large sparse continuous Sylvester equations $AX + XB = C$, where the coefficient matrices $A$ and $B$ are Toeplitz matrices. A theoretical study shows that if the circulant and skew-circulant splitting factors of $A$ and $B$ are positive semi-definite and at least one is positive definite (not necessarily Hermitian), then the CSCS method converges to the unique solution of the Sylvester equation. In addition, we obtain an upper bound for the convergence factor of the CSCS iteration. This convergence factor depends only on the eigenvalues of the circulant and  
skew-circulant splitting matrices.  A computational comparison with alternative methods reveals the efficiency and  reliability of the proposed method.
\end{abstract}


\vskip 10pt

\section{\normalsize\bf Introduction}
A continuous Sylvester equation is possibly  one of the most popular linear matrix equations used in mathematics. It is a matrix equation of the form
\begin{equation}\label{original-equation}
A X + X B = C,
\end{equation}
where matrices $A \in\mathbb{C}^{n\times n}$, $B\in\mathbb{C}^{m\times m}$, $C\in\mathbb{C}^{n\times m}$ are given and the problem is to find a matrix $X\in\mathbb{C}^{n\times m}$ that obeys this equation. It is well-known that equation (\ref{original-equation}) has a unique solution for $X$ if and only if $A$ and $-B$ do not have  common eigenvalues (see, e.g.,
 \cite{Horn1991, Lancaster1985}).

The Sylvester equation is classically employed in the design of Luenberger observers, which are widely used in signal processing, control and system theory (see, e.g., \cite{Benner2009, Bhatia1997, Dooren1991,Golub1996, Kittisopaporn2021}); often appears in linear and generalized eigenvalue problems for the Riccati equation in the computation of invariant subspaces (see, e.g., \cite{Bai2006,Duang2020,Simoncini2016}); can be used to devise implicit Runge-Kutta integral formulae and block multi-step formulae for the numerical solutions of ordinary differential equations (see, e.g., \cite{Epton1980}); and  some linear systems arising, for example, from finite difference discretizations of separable elliptic boundary value problems on rectangular domains, can be written as a Sylvester equation (see, e.g., \cite{Chen2000,ElmanGolub1991}).

There are essentially two different approaches to deal with the Sylvester equation (\ref{original-equation}). The first approach consists in vectorizing the unknown matrix $X$ and translating the matrix equation into a linear system $\mathscr{A} \mathbf{x} = \mathbf{c}$, 
where vectors $\mathbf{x}$ and $\mathbf{c}$ are the column-stacking vectors of the matrices $X$ and $C$, respectively, 
and $\mathscr{A}$ is the \textsl{Kronecker sum} of the matrices $A$ and $B^{T}$, that is, $\mathscr{A}= I_{m} \otimes A + B^{T} \otimes I_{n}$, with symbol $\otimes$ denoting the standard Kronecker product. Either direct or iterative methods can be applied to solve this linear system.
The second approach is to treat the Sylvester equation (\ref{original-equation}) in its original form using an iterative method directely applied to matrices $A$, $B$ and $C$.

When matrices $A$ and $B$ are large, the order of the coefficient matrix $ \mathscr{A} \in\mathbb{C}^{mn\times mn}$ in the linear system $\mathscr{A} \mathbf{x} = \mathbf{c}$ will be considerably larger and, in general, difficulties  related to data storage and computational time arise. This explains that the first approach is mainly used in problems of small or medium dimension. The Bartels-Stewart method proposed in \cite{Bartels1972} is based on the reduction of the matrices $A$ and $B$ to real Schur form (quasi-triangular form) using the QR algorithm for eigenvalues, followed by the use of direct methods to solve several linear systems. The Hessenberg-Schur method, presented in \cite{Golub1979}, reduces matrix $A$ to Hessenberg form and only matrix $B$ is decomposed into the quasi-triangular Shur form ant it is faster then the Bartels-Stewart method. However, in both methods, the authors were unable to establish a  backward stability result. These methods are classsified as direct methods and are used by \textsc{Matlab}.

When matrices $A$ and $B$ are large and sparse, following the second approach, iterative methods such as the Smith's method \cite{Smith},
 the \textsl{alternating direction implicit} method (ADI) \cite{Benner2014, Calvetti1996, Hu1992, Levenberg1993, Wachspress1988}, 
 the \textsl{block successive over-relaxation} method (BSOR) \cite{Starke1991} and the \textsl{matrix splitting} methods \cite{Bai2007,Gu2009} are efficient and accurate methods to obtain a numerical solution of the equation (\ref{original-equation}). The development of the mentioned iterative methods based on the concept of matrix splitting has attracted several scholars and a large number of efficient and robust algorithms were proposed. See \cite{Bai2011, liu2020, liuCarla2020, Wang2013, Zheng2014} and references therein.

In this paper, we consider the case when $A$ and $B$ are both Toeplitz matrices.  Matrices with this structure appear, for example, in connection to the discretization of the convection-diffusion reaction equation \cite{Chen2000,ElmanGolub1991}. We present an iterative method for solving the Sylvester equation (\ref{original-equation}) using  the circulant and skew-circulant splittings of the matrices $A$ and $B$. This circulant and skew-circulant splitting (CSCS) iteration method is a matrix variant of the CSCS iteration method firstly proposed in \cite{Michael2003} for solving a Toeplitz linear system.  These type of methods are conceptually analogous to the ADI iteration methods. Via this CSCS iteration method, the problem of solving a general continuous Sylvester equation is translated into two coupled continuous Sylvester equations involving shifted circulant and skew-circulant matrices. 

When the circulant and skew-circulant splitting matrices of $A$ and $B$ are positive definite (not necessarily Hermitian), we prove that the CSCS iteration converges unconditionally to the exact solution of the  Sylvester equation (\ref{original-equation}).
Moreover, the values of the shift parameters that minimize an upper bound for the contraction factor are obtained in terms of the bounds for the largest and the smallest eigenvalues of the circulant and skew-circulant splitting matrices of $A$ and $B$.

The organization of this paper is as follows. After giving some basic definitions and preliminary results in section 2, we 
describe the CSCS iterative method for solving equation (\ref{original-equation}) in section 3.
We then analyze some sufficient conditions that ensure the convergence of this method in section 4. Numerical experiments are shown in section 5. These examples illustrate the efficiency and robustness of our method.

\section{Basic definitions and preliminary results}
Given a matrix $K\in\mathbb{C}^{n\times m}$, $K^{*}$ denotes the conjugate transpose of $K$, the $(i,j)$ element of $K$ is denoted by $K_{i,j}$ and $\rho(K)$ stands for the spectral radius of $K$. The set of all the eigenvalues of $K$ is represented by $\lambda(K)$. If $x\in\mathbb{C}$, $\operatorname{Re}(x)$ denotes the real part of $x$ and $\operatorname{Im}(x)$ the imaginary part. 

  Here we use the general concept of positive definiteness which says that a matrix $K\in\mathbb{C}^{n\times n}$ is positive definite if its Hermitian part $\frac{1}{2}(K+K^{*})$ is positive definite in the narrower sense. In general, this condition is equivalent to
$\operatorname{Re}(\boldsymbol{z}^*K\boldsymbol{z})>0$, for all nonzero vectors $z \in \mathbb{C}$, which implies that $\operatorname{Re}(\lambda)>0$, for any eigenvalue $\lambda$ of $K$.

A square matrix $T$ is said to be Toeplitz (or diagonal-constant)  when $ T_{j,k} = t_{j-k}$, $j,k=1,\ldots,n$, for constants $t_{1-n}, \ldots, t_{n-1}$. An important property of a Toeplitz matrix $T$ is that it always admits the additive decomposition
\begin{equation}\label{CSsplitting}
T = C_T + S_T,
\end{equation}
where $C_T$ is a  circulant matrix and $S_T$ is a skew-circulant matrix. See \cite{Michael2003,Michael2004}. We say that a matrix $C$ is a circulant matrix if $C_{j,k}=c_{j-k}$, for constants $c_{1-n}, \ldots, c_{n-1}$  such that 
$c_{-l}= c_{n-l}$, $l=1,\ldots,n-1$. That is, a circulant matrix is a Toeplitz matrix that is fully defined by its first column (or row) given that the remaining columns are cyclic permutations of the first column (or row). A skew-circulant matrix is also a particular type of   Toeplitz matrix. We say that a matrix $S$ is a skew-circulant matrix if $S_{j,k}=s_{j-k}$, for constants $s_{1-n}, \ldots, s_{n-1}$  such that $s_{-l}= -s_{n-l}$, $l=1,\ldots,n-1$.

Matrices $C_T$ and $S_T$ in \eqref{CSsplitting}, the circulant and skew-circulant splitting (CSCS) of $T$, are defined as follows:
\begin{equation}\label{splittings}
\begin{array}{lll}
{(C_T)}_{j,k}=
\frac{1}{2}\left\{
\begin{array}{cc}
t_0,                             & {\mbox{\rm if}} \  j=k,   \\[0.2cm]
t_{j-k}+t_{j-k+n} ,                 & {\mbox{\rm if}} \ j<k, \\[0.2cm]
t_{j-k}+t_{j-k-n} ,                 & {\mbox{\rm if}} \ j>k, 
\end{array}
\right. &
\mbox{\rm \quad and \quad } &
{(S_T)}_{j,k}=\frac{1}{2}
\left\{
\begin{array}{cc}
t_0,  & {\mbox{\rm if}} \ j=k,\\[0.2cm]
t_{j-k}-t_{j-k+n}  ,   & {\mbox{\rm if}} \ j<k, \\[0.2cm]
t_{j-k}-t_{j-k-n}  ,                 & {\mbox{\rm if}} \ j>k.
\end{array}
\right.
\end{array}
\end{equation}

%

It is well-known that a circulant matrix $C$ is diagonalizable by the unitary Fourier matrix $F$ of order $n$ which entries are given by 
\[
F_{j,k}=\frac{1}{\sqrt{n}}\omega^{(j-1)(k-1)}, \quad j,k=1,\ldots,n,
\]
where $\omega$ is the primitive $n$-th root of unit $\omega=\operatorname{\boldsymbol{e}}^{\frac{2\pi}{n}\textbf{i}}$, $\textbf{i}=\sqrt{-1}$. Similarly, a skew-circulant matrix $S$ is diagonalizable by the unitary matrix $\hat{F}=FD$ where
$D=\operatorname{diag}(1,\operatorname{\boldsymbol{e}}^{\frac{\pi }{n}\textbf{i}},
\ldots,\operatorname{\boldsymbol{e}}^{\frac{(n-1)\pi}{n}\textbf{i}})$. Thus,
\begin{equation}\label{FourierCS}
C=F^*\Lambda F \qquad \text{and} \qquad S=\hat{F}^* \Sigma  \hat{F},
\end{equation}
where $\Lambda$ and $\Sigma$ are diagonal matrices holding the eigenvalues of $C$ and $S$, respectively.  
Moreover, we note that $\Lambda$ and $\Sigma$ can be obtained in ${\cal{O}}(n \log n)$ operations by taking the fast Fourier transform
(FFT) of the first column (or row) of $C$ and first row of $S$, respectively.  In fact, the diagonal entries $\lambda_j$ of $\Lambda$ and the diagonal entries $\sigma_j$ of $\Sigma$ are given, respectively, by 
\begin{equation}\label{explicitEigFormlae}
\lambda_j=\sum_{k=1}^nc_{k-1}\omega^{(j-1)(k-1)} \qquad \text{and} \qquad 
\sigma_j=\sum_{k=1}^ns_{k-1}\omega^{(j-1)(k-1)}\operatorname{\boldsymbol{e}}^{\frac{\pi(k-1)}{n}\textbf{i}}, \quad j=1,\ldots, n.
\end{equation}
See, for instance, \cite{ Chan1996,Davis1979}. For the FFT algorithm, we refer to 
\cite{Loan1992,Vetterli1984}.

Once $\Lambda$ and $\Sigma$ are obtained, the products $C\boldsymbol{x}$ and $C^{-1}\boldsymbol{x}$, as well as
$S\boldsymbol{x}$ and $S^{-1}\boldsymbol{x}$, for any vector $\boldsymbol{x}$, can be computed by FFTs in ${\cal{O}}(n \log n)$ operations. Therefore, the use of circulant and skew-circulant matrices to solve matrix equations with Toeplitz matrices allows to improve the efficiency by employing FFTs throughout the computations. For a matrix $M\in {\mathbb{C}}^{n\times m}$, the FFT operation is applied to each column in ${\cal{O}}(mn \log n)$.

\section{\normalsize\bf CSCS Iteration}
Let matrices $A \in\mathbb{C}^{n\times n}$ and  $B\in\mathbb{C}^{m\times m}$ have a Toeplitz structure and let
\begin{equation}\label{CACB}
A=C_{A}+S_{A}, \quad  B=C_{B}+S_{B}
\end{equation}
be the circulant and skew-circulant splittings of $A$ and $B$ (CSCS), respectively. There are no constraints in using \eqref{splittings}, so these splittings always exist.

If $\alpha$ and $\beta$ are positive constants, the following splittings are also CSCS splittings of $A$ and $B$,
\begin{align}\label{splittinfAB}
A&=(C_{A}+\alpha I_n)+(S_{A}-\alpha I_n),\quad B=(C_{B}+\beta I_m)+(S_{B}-\beta I_m),\\
A&= (C_{A}-\alpha I_n)+(S_{A}+\alpha I_n), \quad B=(C_{B}-\beta I_m)+(S_{B}+\beta I_m).
\end{align}
It follows that if $X^*\in\mathbb{C}^{n\times m}$ is the exact solution of the Sylvester equation (\ref{original-equation}), then  
\begin{equation*}
\left\{\begin{aligned}
( C_{A}+\alpha I) X^*+X^*(C_{B}+\beta I) &= (\alpha I - S_{A})X^*+X^*(\beta I-S_{B}) + C\\
(S_{A}+\alpha I)X^*+X^*(S_{B}+\beta I) &= (\alpha I - C_{A})X^*+X^*(\beta I-C_{B}) + C
\end{aligned}
\right.
\end{equation*}
where $I$ is either the identity matrix of order $n$ or $m$, conformable to $C_A$ or $C_B$, respectively. We are now able to define the fixed-point matrix equations
\begin{equation}
\left\{\begin{aligned}
( C_{A}+\alpha I ) X+X(C_{B}+\beta I) &= (\alpha I - S_{A})Y+Y(\beta I-S_{B}) + C\\
( S_{A}+\alpha I ) Y+Y(S_{B}+\beta I) &= (\alpha I - C_{A})X+X(\beta I-C_{B}) + C
\end{aligned}
\right.
\label{fixedPointEqs}
\end{equation}
such that $X^*$ is the fixed point of both equations. The reverse also accurs, that is, if 
$X^*$ is a fixed point of either of the two equations in \eqref{fixedPointEqs}, then it is the exact solution of \eqref{original-equation}.
See \cite[Theorem 3.1]{BaiGolub2007}.

The CSCS iteration method is defined as follows. 

\medskip

\par\noindent\textbf{CSCS iteration method. }
{\it{Given an initial approximation $X^{(0)}$ and positive constants $\alpha$, $\beta$\break(shift parameters), repeat the iterative scheme
		\begin{equation}\label{CSCSIteration}
		\begin{cases}
		(\alpha I + C_{A})X^{(k+\frac{1}{2})}+X^{(k+\frac{1}{2})}(\beta I+C_{B}) = (\alpha I - S_{A})X^{(k)}+X^{(k)}(\beta I-S_{B}) + C
		\qquad \left(\text{solve for }{X}^{(k+\frac{1}{2})}\right)\\
		(\alpha I + S_{A})X^{(k+1)}+X^{(k+1)}(\beta I+S_{B}) = (\alpha I - C_{A})X^{(k+\frac{1}{2})}+X^{(k+\frac{1}{2})}(\beta I-C_{B}) + C\;\,   \left(\text{solve for }{X}^{(k+1)}\right)
		\end{cases}
		\end{equation}
		for $k=0,1,2,\cdots, $ until $\left\{X^{(k)}\right\}$ converges.} }

An alternative version of this two-step iteration is obtained if we compute the corrections $Z^{(k+\frac{1}{2})}=X^{(k+\frac{1}{2})}-X^{(k)}$ and 
$Z^{(k+1)}=X^{(k+1)}-X^{(k+\frac{1}{2})}$ in each iteration which brings the residuals $R^{(k)}$ and $R^{(k+\frac12)}$
into the computation. The iterative scheme is changed to
\begin{equation}\label{CSCSIteration2}
\begin{cases}
R^{(k)}=C-AX^{(k)}-X^{(k)}B\\
(\alpha I + C_{A})Z^{(k+\frac{1}{2})}+Z^{(k+\frac{1}{2})}(\beta I+C_{B}) = R^{(k)}
\qquad \qquad \left(\text{solve for }{Z}^{(k+\frac{1}{2})}\right)\\
X^{(k+\frac{1}{2})}=X^{(k)}+Z^{(k+\frac{1}{2})}\\
R^{(k+\frac12)}=C-AX^{(k+\frac12)}-X^{(k+\frac12)}B\\
(\alpha I + S_{A})Z^{(k+1)}+Z^{(k+1)}(\beta I+S_{B}) = R^{(k+\frac12)} \;\,  \quad \qquad \left(\text{solve for }{Z}^{(k+1)}\right)\\
X^{(k+1)}=X^{(k+\frac12)}+Z^{(k+1)}\\
\end{cases}
\end{equation}
Each CSCS iteration requires the solution of two Sylvester equations. In the first step ${Z}^{(k+\frac{1}{2})}$ is the solution of the equation
	\begin{equation}\label{matrixEqu1}
		(\alpha I + C_{A})Z+Z(\beta I+C_{B})=R^{(k)}
	\end{equation}
and in the second step $Z^{(k+1)}$ is the solution of the equation
	\begin{equation}\label{matrixEqu2}
(\alpha I + S_{A})Z +Z (\beta I+S_{B})=R^{(k+\frac{1}{2})}.
	\end{equation}
The method alternates between equation $\eqref{matrixEqu1}$, with circulant matrices $C_A$ and $C_B$, and equation  $\eqref{matrixEqu2}$,  with skew-circulant matrices $S_A$ and $S_B$, and we can reverse the roles of these matrix equations.

Observe that once we compute the eigenvalues of $C_{A}$ and $C_B$, it is always possible to choose positive constants $\alpha$ and $\beta$ such that $\alpha I + C_{A}$ and $-(\beta I+C_{B})$ do not have eigenvalues in common and thus the matrix equation $\eqref{matrixEqu1}$ has a unique solution. A similar observation applies to the matrix equation $\eqref{matrixEqu2}$ concerning the matrices $\alpha I + S_{A}$ and $-(\beta I+S_{B})$.

Obtained the eigenvalue decompositions, see \eqref{FourierCS},
\begin{align}\label{FourierFac}
C_A=F_A^*\Lambda_A F_A, \qquad C_B=F_B^*\Lambda_B F_B, \qquad S_A=\hat{F}_A^* \Sigma_A  \hat{F}_A, \qquad S_B=\hat{F}_B^* \Sigma_B  \hat{F}_B,
\end{align}
equation $\eqref{matrixEqu1}$ is equivalent to  
\begin{equation}\label{matrixEqu1a}
(\alpha I+\Lambda_A){\cal{Z}}+{\cal{Z}}(\beta I+\Lambda_B)={\cal{R}}^{(k)}
\end{equation}
where ${\cal{Z}}=F_AZF_B^*$ and ${{\cal{R}}^{(k)}}=F_AR^{(k)}F_B^*$;
and equation $\eqref{matrixEqu2}$ is equivalent to
\begin{equation}\label{matrixEqu2a}
 (\alpha I+\Sigma_A){\cal{Z}}+{\cal{Z}}(\beta I+\Sigma_B)={{\cal{R}}^{(k+\frac{1}{2})}}
\end{equation}
where ${\cal{Z}}=\hat F_AZ\hat F_B^*$ and ${{\cal{R}}^{(k+\frac{1}{2})}}=\hat F_AR^{(k+\frac{1}{2})}\hat F_B^*$.

Solving the Sylvester equations \eqref{matrixEqu1a} and \eqref{matrixEqu2a} is immediate since these matrix equations can be translated into linear systems with diagonal coefficent matrices, $I_{m} \otimes (\alpha I_n+\Lambda_A) +(\beta I_m+\Lambda_B)^{T} \otimes I_{n}$ and 
$I_{m} \otimes (\alpha I_n+\Sigma_A) + (\beta I_m+\Sigma_B)^{T} \otimes I_{n}$, respectively. 

The solutions of the initial Sylvester equations \eqref{matrixEqu1} and \eqref{matrixEqu2} are given by $Z={F_A^*}{\cal{Z}}F_B$, for ${\cal{Z}}$ satisfying \eqref{matrixEqu1a},  and $Z=\hat F_A^*{\cal{Z}}\hat F_B$, for ${\cal{Z}}$ satisfying \eqref{matrixEqu2a}, respectively. These products can be computed efficiently using FFTs.

It is also possible, once again using FFTs, to reduce the computational effort associated to the right-hand side of the  two Sylvester equations involved in each CSCS iteration, equations \eqref{matrixEqu1a} and \eqref{matrixEqu2a}. For an approximation $X^{(j)}$, 
the residual $R^{(j)}$ can be computed using the decomposition 
\begin{align}
{R^{(j)}}&=C-\left(C_A+S_A\right)X^{(j)}-X^{(j)}(C_B+S_B)\nonumber\\
&=C-F_A^*\Lambda_{A}F_AX^{(j)}-\hat{F}_A^*\Sigma_A \hat{F}_AX^{(j)}
-X^{(j)}F_B^*\Lambda_{B}F_B-X^{(j)}\hat{F}_B^*\Sigma_B\hat{F}_B.\label{residualGeneral}
\end{align}
Thus, the right-hand sides of 
equations \eqref{matrixEqu1a} and \eqref{matrixEqu2a}  are given, respectively,  by
\begin{align}
{{\cal{R}}^{(k)}}&=F_A\left[C-F_A^*\Lambda_{A}F_AX^{(k)}-\hat{F}_A^*\Sigma_A \hat{F}_AX^{(k)}
-X^{(k)}F_B^*\Lambda_{B}F_B-X^{(k)}\hat{F}_B^*\Sigma_B\hat{F}_B\right]F_B^*\label{residual1}
\end{align}
and   
\begin{align}
{{\cal{R}}^{(k+\frac{1}{2})}}
&=\hat F_A\left[C-F_A^*\Lambda_{A}F_AX^{(k+\frac{1}{2})}-\hat{F}_A^*\Sigma_A \hat{F}_AX^{(k+\frac{1}{2})}
-X^{(k+\frac{1}{2})}F_B^*\Lambda_{B}F_B-X^{(k+\frac{1}{2})}\hat{F}_B^*\Sigma_B\hat{F}_B\right]\hat F_B^*.\label{residual2}
\end{align}

Given an initial approximation $X^{(0)}$, positive constants $\alpha$, $\beta$, and the spectral factorizations \eqref{FourierFac}, the two steps of the CSCS iteration \eqref{CSCSIteration2} can then be expressed as
\begin{equation}\label{NewScheme}
\begin{cases}
\text{Use }\eqref{residual1} \text{ to compute } {{\cal{R}}^{(k)}} \\
(\alpha I+\Lambda_A){\cal{Z}}+{\cal{Z}}(\beta I+\Lambda_B)={\cal{R}}^{(k)}
\hspace{4.5cm}\left(\text{solve for }{\cal Z}\right)\\
X^{(k+\frac{1}{2})}=X^{(k)}+{F_A^*}{\cal Z}F_B\\
\text{Use }\eqref{residual2} \text{ to compute }  {{\cal{R}}^{(k+\frac12)}} \\
 (\alpha I+\Sigma_A){\cal{Z}}+{\cal{Z}}(\beta I+\Sigma_B)={{\cal{R}}^{(k+\frac{1}{2})}} \;\,  
 \hspace{4cm} \left(\text{solve for }{\cal Z}\right)\\
X^{(k+1)}=X^{(k+\frac12)}+\hat F_A^*{\cal Z}\hat F_B\\	
\end{cases}
\end{equation}
for $k=0,1,2,\cdots, $ until $\left\{X^{(k)}\right\}$ converges.

Notice that all the matrix multiplications can be performed using FFTs and thus the operation count is ${\cal{O}}(mn \log n)$ (or ${\cal{O}}(nm \log m)$, depending on which is bigger). There is no need to perform explicit matrix multiplications.
See Algorithm \ref{alg:algorithm1} in Appendix \ref{ApendiceA} for a \textsc{Matlab} implemention of CSCS.

\section{\normalsize\bf Convergence results} Using a matrix-vector formulation of the Sylvester equation \eqref{original-equation}, such that vectors $\mathbf{x}$ and $\mathbf{c}$ are the column-stacking vectors of the matrices $X$ and $C$, respectively, the two-step iterative CSCS scheme \eqref{CSCSIteration} can be rewritten as
\begin{equation}\label{CSCSform2}
\begin{cases}
\left[I_{m} \otimes (\alpha I + C_{A}) + (\beta I+C_{B})^{T} \otimes I_{n}\right]\mathbf{x}^{(k+\frac{1}{2})}=
\left[I_{m} \otimes (\alpha I - S_{A}) + (\beta I-S_{B})^{T} \otimes I_{n}\right]\mathbf{x}^{(k)}+\mathbf{c}\\ 
\left[I_{m} \otimes (\alpha I + S_{A}) + (\beta I+S_{B})^{T} \otimes I_{n}\right]\mathbf{x}^{(k+1}=
\left[I_{m} \otimes (\alpha I - C_{A}) + (\beta I-C_{B})^{T} \otimes I_{n}\right]\mathbf{x}^{(k+\frac{1}{2})}+\mathbf{c}\\ 
\end{cases}
\end{equation}
for $k=0,1,2,\cdots, $ until $\left\{\mathbf{x}^{(k)}\right\}$ converges, given an initial approximation $\mathbf{x}^{(0)}$ and positive constants $\alpha$, $\beta$. 

In this section we will establish  theoretical results concerning the convergence conditions of the CSCS iteration and our analysis is based on this matrix-vector formulation of the method. 

So, we are considering the linear system  $\mathscr{A} \mathbf{x} = \mathbf{c}$, equivalent to the Sylvester equation \eqref{original-equation}, where  $\mathscr{A}= I_{m} \otimes A + B^{T} \otimes I_n$, and the two splittings of the matrix $\mathscr{A}$,
\begin{align*}
\mathscr{A}&=\left[I_{m} \otimes (\alpha I + C_{A}) + (\beta I+C_{B})^{T} \otimes I_{n}\right]-
\left[I_{m} \otimes (\alpha I - S_{A}) + (\beta I-S_{B})^{T} \otimes I_{n}\right],\\
\mathscr{A}&=\left[I_{m} \otimes (\alpha I + S_{A}) + (\beta I+S_{B})^{T} \otimes I_{n}\right]-
\left[I_{m} \otimes (\alpha I - C_{A}) + (\beta I-C_{B})^{T} \otimes I_{n}\right],
\end{align*}
corresponding to the CSCS splittings of $A$ and $B$ given in \eqref{splittinfAB}. Using the bilinearity property of the kronecker product, these decompositions of $\mathscr{A}$ can be rewritten as
\begin{align*}
\mathscr{A}&=\left[(\alpha +\beta)I_{mn} +\left(I_m\otimes  C_{A}+ C_{B}^{T} \otimes I_{n}\right)\right]-
\left[(\alpha +\beta)I_{mn}- \left(I_m \otimes S_{A}+S_{B}^{T} \otimes I_{n}\right)\right],\\
\mathscr{A}&=\left[(\alpha +\beta)I_{mn} +\left(I_m\otimes  S_{A}+ S_{B}^{T} \otimes I_{n}\right)\right]-
\left[(\alpha +\beta)I_{mn}- \left(I_m \otimes C_{A} +C_{B}^{T} \otimes I_{n}\right)\right].
\end{align*}
Thereby, defining  $\widetilde{C}= \left(I_m\otimes  C_{A}+ C_{B}^{T} \otimes I_{n}\right)$ and $\widetilde{S}=\left(I_m \otimes S_{A}+S_{B}^{T} \otimes I_{n}\right)$, we have
\begin{align} 
\begin{aligned}
\mathscr{A}&=\big[(\alpha +\beta)I +\widetilde{C}\big]-
\big[(\alpha +\beta)I- \widetilde{S}\big],\\
\mathscr{A}&=\big[(\alpha +\beta)I +\widetilde{S}\big]-
\big[(\alpha +\beta)I- \widetilde{C}\big],
\end{aligned}
\label{newformdecomp}
\end{align}
where $I$ represents the identity matrix of order $mn$, and the iteration \eqref{CSCSform2} can be expressed as 
\begin{equation}\label{NewIterScheme}
\begin{cases}
\big[(\alpha +\beta)I +\widetilde{C}\big]\mathbf{x}^{(k+\frac{1}{2})}=
\big[(\alpha +\beta)I- \widetilde{S}\big]\mathbf{x}^{(k)}+\mathbf{c}\\ 
\big[(\alpha +\beta)I +\widetilde{S}\big]\mathbf{x}^{(k+1}=
\big[(\alpha +\beta)I- \widetilde{C}\big]\mathbf{x}^{(k+\frac{1}{2})}+\mathbf{c}.\\ 
\end{cases}
\end{equation}
This iterative scheme is a particular case of a more general two-step splitting scheme defined by
\begin{equation} 
\begin{cases}
M_{1}\mathbf{x}^{(k+\frac{1}{2})}=N_{1}\mathbf{x}^{(k)}+\mathbf{c}\\
M_{2}\mathbf{x}^{(k+1)}=N_{2}\mathbf{x}^{(k+\frac{1}{2})}+\mathbf{c},  \qquad  k=0,1,2,\ldots, \\
\end{cases}
\label{generalscheme}
\end{equation}
assuming that $\mathscr{A}$ admits the decompositions $\mathscr{A}= M_{i}-N_{i}$, $i = 1,2$, with $M_1$ and $M_2$ invertible matrices.\break The sequence $\left\{\mathbf{x}^{(k)}\right\}$ generated by \eqref{generalscheme} satisfies 
\begin{equation}\label{twoformulae0}
\mathbf{x}^{(k+1)} = \mathscr{M}\mathbf{x}^{(k)} + \mathscr{C} \mathbf{c}, \qquad k=0,1,2,\ldots,
\end{equation}
where  
\begin{equation}\label{twoformulae}
\mathscr{M}= M_{2}^{-1} N_{2} M_{1}^{-1} N_{1}  \text{\quad and \quad } \mathscr{C}=M_{2}^{-1}\left(I + N_{2} M_{1}^{-1}\right).
\end{equation}
Matrix $\mathscr{M}$ is called the {\textit{iteration matrix}} 
and it is well-kown  that $\left\{\mathbf{x}^{(k)}\right\}$ converges to the exact solution of the linear system $\mathscr{A} \mathbf{x} = \mathbf{c}$ if and only if $\rho(\mathscr{M})<1$, for any initial approximation $\mathbf{x}^{(0)}\in \mathbb{C}^{mn}$ \cite{Saad00, DavidYoung} .   

We will prove that when $\widetilde{C}=I_m\otimes C_{A}+C_{B}^{T}\otimes I_n \ \mbox{and} \ \widetilde{S}=I_m\otimes S_{A}+S_{B}^{T}\otimes I_n$ are positive definite matrices (the real part of the eigenvalues is positive), then the proposed CSCS iterative method converges.

\begin{theorem}\label{CSCS convergence 1}
Let $A\in\mathbb C^{n\times  n}$ and $B\in\mathbb C^{m\times m}$ be Toeplitz matrices such that $A = C_{A} + S_{A}$ and $B = C_{B} + S_{B}$ are the circulant and skew-circulant splittings of $A$ and $B$, respectively.  Consider the linear system 
$\mathscr{A} \mathbf{x} = \mathbf{c}$, where  $\mathscr{A}= I_{m} \otimes A + B^{T} \otimes I_n$, equivalent to the Silvester equation \eqref{original-equation}, and the splitting 
$
\mathscr{A}=\widetilde{C}+\widetilde{S},
$
where
 \begin{equation}\label{CtilStil}
 \widetilde{C}=I_m\otimes C_{A}+C_{B}^{T}\otimes I_n \ \mbox{\quad and \quad } \ \widetilde{S}=I_m\otimes S_{A}+S_{B}^{T}\otimes I_n.
 \end{equation}
Let $\alpha$, $\beta$ be two positive constants and  $\gamma =\alpha+\beta$. Then the iteration matrix of the CSCS scheme \eqref{NewIterScheme} is  
\begin{equation}\label{iterMatrixM}
\mathscr{M}_\gamma=\big(\gamma I+\widetilde{S}\big)^{-1}\big(\gamma I-\widetilde{C}\big)\big(\gamma I+\widetilde{C}\big)^{-1}\big(\gamma I-\widetilde{S}\big)
\end{equation}
and its spectral radius $\rho(\mathscr{M}_\gamma)$ is bounded by
$$
\sigma_\gamma\equiv \max\limits_{\lambda_{j} \in \lambda(\tilde{C})} {\left|  \frac{\gamma - \lambda_{j}}{\gamma + \lambda_{j}} \right| }\cdot\max\limits_{\mu_{j} \in \lambda(\tilde{S})} {\left|  \frac{\gamma- \mu_{j}}{\gamma + \mu_{j}} \right|}.$$
 If $\widetilde{C}$ is positive definite and $\widetilde{S}$ is positive semi-definite (or vice-versa), then 
\[
\rho(\mathscr{M}_\gamma)\leq \sigma_\gamma<1, \quad \text{for all }\gamma >0,
\]
and, thus, the CSCS iteration \eqref{NewIterScheme} converges to the exact solution $\mathbf{x}^{\star}$ of the linear system $\mathscr{A} \mathbf{x} = \mathbf{c}$.\break The equivalent CSCS iteration \eqref{CSCSIteration} converges to the exact solution $X^{\star}\in\mathbb C^{m \times n}$
of the  Sylvester equation (\ref{original-equation}). 
\end{theorem}

\begin{proof} Given $\gamma =\alpha+\beta$, the CSCS iteration \eqref{NewIterScheme} can be rewritten as 
\begin{equation}\label{phi23}
\begin{cases}
\big(\gamma I+\widetilde{C}\big)\mathbf{x}^{(k+\frac{1}{2})}=\big(\gamma I-\widetilde{S}\big)\mathbf{x}^{(k)}+\mathbf{c} \\
\big(\gamma I+\widetilde{S}\big)\mathbf{x}^{(k+1)}=\big(\gamma I-\widetilde{C}\big)\mathbf{x}^{(k+\frac{1}{2})}+\mathbf{c}
\end{cases}
\end{equation}
 which is the two-step splitting iterative scheme \eqref{generalscheme} to solve the linear system $\mathscr{A} \mathbf{x} = \mathbf{c}$ with $M_1=\gamma I+\widetilde{C}$, $N_1=\gamma I-\widetilde{S}$, $M_2=\gamma I+\widetilde{S}$ and $N_2=\gamma I-\widetilde{C}$. According to 
 \eqref{twoformulae0} and \eqref{twoformulae}, the two steps of iteration \eqref{phi23} can be put together into the stationary fixed-point iteration 
 \begin{equation}\label{fixedPoint2}
 \mathbf{x}^{(k+1)} = \mathscr{M}_{\gamma}\mathbf{x}^{(k)} + \mathscr{C}_{\gamma} \mathbf{c}, \qquad k=0,1,2,\ldots,
 \end{equation}
 where 
 \[
 \mathscr{M}_{\gamma}=\big(\gamma I+\widetilde{S})^{-1}\big(\gamma I-\widetilde{C}\big)\big(\gamma I+\widetilde{C}\big)^{-1}\big(\gamma I-\widetilde{S}\big) \ \text{\quad and \quad} \
  \mathscr{C}_{\gamma} =2\gamma\big(\gamma I+\tilde{S}\big)^{-1}\big(\gamma I+\tilde{C}\big)^{-1}.
 \]
The spectral radius $\rho( \mathscr{M}_{\gamma})$ governs the convergence of \eqref{fixedPoint2} and, since the spectrum of a matrix is invariant under a similarity transformation, we find that 
\begin{equation}
\begin{aligned}
\rho(\mathscr{M}_{\gamma})&=\rho\left(\big(\gamma I-\widetilde{C}\big)\big(\gamma I+\widetilde{C}\big)^{-1}\big(\gamma I-\widetilde{S}\big)\big(\gamma I+\widetilde{S}\big)^{-1}\right)\\
   &\leq {\left\| \big(\gamma I-\widetilde{C}\big)\big(\gamma I+\widetilde{C}\big)^{-1}\big(\gamma I-\widetilde{S}\big)\big(\gamma I+\widetilde{S}\big)^{-1} \right\|}_{2}\\
               &\leq {\left\|\big(\gamma I-\widetilde{C}\big)\big(\gamma I+\widetilde{C}\big)^{-1}\right\|}_{2}\cdot{\left\|\big(\gamma I-\widetilde{S}\big)\big(\gamma I+\widetilde{S}\big)^{-1}\right\|}_{2}.
\end{aligned}
\label{ineqNorm}
\end{equation}
Since $C_A$, $C_B$, $F_A$ and $F_B$ are diagonalizable by the Fourier-type matrices $F_A$, $F_B$, $\hat{F}_A$ and $\hat{F}_B$, respectively, (see \eqref{FourierCS}), then $\widetilde{C}$ is diagonalizable by $F_B\otimes F_A$  and $\widetilde{S}$ by $\hat{F}_B\otimes\hat{F}_A$ \cite{Horn1991},
\begin{align}
\label{diagonalizations}
\widetilde{C}=(F_B\otimes F_A)^*\Lambda_{\widetilde{C}}(F_B\otimes F_A), \quad 
\widetilde{S}=(\hat F_B\otimes \hat F_A)^*\Sigma_{\widetilde{S}}(\hat F_B\otimes \hat F_A),
\end{align}
where $\Lambda_{\widetilde{C}}=I\otimes \Lambda_{A}+\Lambda_{B}\otimes I$ and $\Sigma_{\widetilde{S}}=I\otimes \Sigma_{A}+\Sigma_{B}\otimes I$. 

Thus, by \eqref{diagonalizations}  and the invariance of the matrix 2-norm under a unitary similarity,
\begin{align*}
{\left\|\big(\gamma I-\widetilde{C}\big)\big(\gamma I+\widetilde{C}\big)^{-1}\right\|}_{2}&=
{\left\|\big(\gamma I-\Lambda_{\widetilde{C}}\big)\big(\gamma I+\Lambda_{\widetilde{C}}\big)^{-1}\right\|}_{2}=
\max\limits_{ \lambda_{k} \in \lambda(\widetilde{C})} {\left|  \frac{\gamma - \lambda_k}{\gamma + \lambda_k} \right|},\\
{\left\|\big(\gamma I-\widetilde{S}\big)\big(\gamma I+\widetilde{S}\big)^{-1}\right\|}_{2}&=
{\left\|\big(\gamma I-\Sigma_{\widetilde{S}}\big)\big(\gamma I+\Sigma_{\widetilde{S}}\big)^{-1}\right\|}_{2}
=\max\limits_{ \mu_{k} \in \lambda(\widetilde{S})} {\left|  \frac{\gamma - \mu_k}{\gamma + \mu_k} \right|. }
\end{align*}
For $\lambda_k=a_k+\mathbf{i}b_k\in \lambda(\widetilde{C})$ and $\mu_k=c_k+\mathbf{i}d_k\in \lambda(\widetilde{S})$,  $k=1,\ldots,mn$, we have $a_k> 0$, $c_k\ge0$ (or $a_k\ge 0$, $c_k>0$), by the assumption that 
$\widetilde{C}$ is positive definite and $\widetilde{S}$ is positive semi-definite (or vice-versa),
and  then
\[
{\left|  \frac{\gamma - \lambda_k}{\gamma + \lambda_k} \right| }=\sqrt{\frac{(\gamma - a_k)^2+b_k^2}{(\gamma + a_k)^2+b_k^2}}<1 \; (\le1),\qquad 
{\left|  \frac{\gamma - \mu_k}{\gamma + \mu_k} \right| }=\sqrt{\frac{(\gamma - c_k)^2+d_k^2}{(\gamma + c_k)^2+d_k^2}}\le1 \; (<1),
\]
since $\gamma >0$. As a consequence,
\begin{equation}\label{sigmaDef}
\sigma_\gamma:=\max\limits_{ \lambda_{k} \in \lambda(\widetilde{C})} {\left|  \frac{\gamma - \lambda_k}{\gamma + \lambda_k} \right|}\cdot \max\limits_{ \mu_{k} \in \lambda(\widetilde{S})} {\left|  \frac{\gamma - \mu_k}{\gamma + \mu_k} \right|}<1.
\end{equation}
Finally, \eqref{ineqNorm} yields
\[
\rho(\mathscr{M}_{\gamma})\leq \sigma_\gamma<1
\]
which ensures that the CSCS iteration \eqref{NewIterScheme} converges to the exact solution $\mathbf{x}^{\star}$ of the linear system $\mathscr{A} \mathbf{x} = \mathbf{c}$ and that the CSCS iteration \eqref{CSCSIteration}, which is equivalent to \eqref{NewIterScheme},  converges to the exact solution $X^{\star}\in\mathbb C^{m \times n}$
of the  Sylvester equation (\ref{original-equation}). 
\end{proof}

\begin{corollary}\label{corollary1}
	 If one of the matrices $C_A$, $C_B$, $S_A$ and $S_B$ is positive definite and all the others are positive semi-definite, then 
	the CSCS iteration \eqref{CSCSIteration} converges to the exact solution $X^{\star}\in\mathbb C^{m \times n}$
	of the  Sylvester equation (\ref{original-equation}). 
\end{corollary}

\begin{proof}
Recall that given two matrices $G\in \mathbb{C}^{n \times n}$ and $H\in \mathbb{C}^{m \times m}$  with eigenvalues $\lambda_i$,  $i=1,\ldots,n$, and $\mu_j,  j=1,\ldots, m$, respectively, the eigenvalues of the Kronecker sum $G\oplus H=I_m\otimes G+H\otimes I_n$ are the 
pairwise sums $\lambda_i+\mu_j$, $i=1,\ldots, n, j=1,\ldots,m$ (see, e.g., \cite{Horn1991}). Using this property, it is immediate to conclude, for example, that if $G$ is positive definite and $H$ is positive semi-definite, then $G\oplus H$ is positive definite. In fact, 
$\operatorname{Re}(\lambda_j)>0$ and $\operatorname{Re}(\mu_j)\ge0$ imply that $\operatorname{Re}(\lambda_i+\mu_j)=\operatorname{Re}(\lambda_i)+\operatorname{Re}(\mu_j)>0$.

Suppose that  $C_A$ is positive definite and $C_B$, $S_A$ and $S_B$ are all positive semi-definite. Then $\widetilde{C}= C_{A}\oplus C_{B}^{T}$ and $\widetilde{S}= S_{A}\oplus S_{B}^{T}$, defined in \eqref{CtilStil}, are positive definite and positive semi-definite, respectively. According to Theorem \ref{CSCS convergence 1}, the spectral radius $\rho(\mathscr{M}_\gamma)$ of the iteration matrix $\mathscr{M}_\gamma$ in \eqref{iterMatrixM} is less than 1 and the CSCS iteration \eqref{CSCSIteration} converges. All the other cases are similar. \end{proof}

Next theorem addresses the issue of how to obtain a value for $\gamma=\alpha+\beta$ (and naturally for $\alpha$ and $\beta$) that leads to a good convergence speed.

\begin{remark}\label{remarkTetaEta}\rm
Let $\lambda_k=a_k+\mathbf{i}b_k\in \lambda(\widetilde{C})$ and $\mu_k=c_k+\mathbf{i}d_k\in \lambda(\widetilde{S})$, $k=1,\ldots,mn$, satisfy
\begin{equation}\label{thetamin}
\theta_{min}\leq a_k,c_k\leq\theta_{max} \qquad \text{and} \qquad \eta_{min}\leq|b_k|,|d_k|\leq\eta_{max},
\end{equation}
where $\theta_{min}$ and $\theta_{max}$ are the lower and upper bounds, respectively, of the real part of the eigenvalues 
$\lambda(\widetilde{C})\cup \lambda(\widetilde{S})$,  and $\eta_{min}$ and $\eta_{max}$ are the lower and the upper bounds, respectively, of the absolute values
 of the imaginary part of the eigenvalues $\lambda(\widetilde{C})\cup \lambda(\widetilde{S})$. A bound for $\sigma_\gamma$ is given by
\begin{equation}\label{estimate}
\max\limits_{(\theta,\eta) \in\Omega}\frac{(\gamma-\theta)^{2}+\eta^{2}}{(\gamma+\theta)^{2}+\eta^{2}}
\end{equation}
where $\Omega=[\theta_{min},\theta_{max}]\times[\eta_{min},\eta_{max}]$, since
\begin{equation*}
\begin{aligned}
\sigma_\gamma\leq \max\limits_{ \lambda_{k} \in \lambda(\widetilde{C})\cup \lambda(\widetilde{S})} {\left|  
	\frac{\gamma - \lambda_k}{\gamma + \lambda_k} \right|}\cdot \max\limits_{ \mu_{k} \in \lambda(\widetilde{C})\cup \lambda(\widetilde{S})} {\left|  \frac{\gamma - \mu_k}{\gamma + \mu_k} \right|}
=\max\limits_{ \lambda_{k} \in \lambda(\widetilde{S})\cup \lambda(\widetilde{C})} {\left|  
	\frac{\gamma - \lambda_k}{\gamma + \lambda_k} \right|^2}.
\end{aligned}
\end{equation*}
\end{remark}

We may consider that the optimal choice $\gamma^{\star}$ for the shift parameter $\gamma$ is the value that minimizes the above estimate \eqref{estimate}.
The following theorem gives an explict formula for $\gamma^{\star}$, if $\theta_{min}>0$.

\begin{theorem}\label{theo:gammaStar}
If $\theta_{min}\ge0$, the  minimum value
$$
\min\limits_{\gamma>0}\left\{\max\limits_{(\theta,\eta) \in\Omega}\frac{(\gamma-\theta)^{2}+\eta^{2}}{(\gamma+\theta)^{2}+\eta^{2}}\right\} 
$$
is attained at
\begin{equation}\label{gammaStar}
\gamma^{\star} = \begin{cases}
\sqrt{\theta_{min}\theta_{max}-\eta_{max}^{2}}  & \text{\rm for} ~~\eta_{max}< \tilde{\eta}\\[12pt]
\sqrt{\theta_{min}^{2}+\eta_{max}^{2}}
& \text{\rm for} ~~ \eta_{max}\ge\tilde{\eta},
\end{cases}
\end{equation}
and it is equal to
$$
\sigma^{\star}=\begin{cases}\dfrac{\theta_{min}+\theta_{max}-2\sqrt{\theta_{min}\theta_{max}-\eta_{max}^{2}}}{\theta_{min}+\theta_{max}+2\sqrt{\theta_{min}\beta_{max}-\eta_{max}^{2}}}
& \text{\rm for}  ~~\eta_{max}< \tilde{\eta},\\[12pt]
\dfrac{\sqrt{\theta_{min}^{2}+\eta_{max}^{2}} -\theta_{min}}{\sqrt{\theta_{min}^{2}+\eta_{max}^{2}} +\theta_{min}}
& \text{\rm for} ~~ \eta_{max}\ge\tilde{\eta}.
\end{cases}
$$
where $\tilde{\eta}=\sqrt{\theta_{min}(\theta_{max}-\theta_{min})/2}$.
\end{theorem}

The proof of this theorem can be found in \cite[pp. 324--326]{BaiGolub2007} and \cite{Bai2012}.

 Concerning the choice of the shift parameters $\alpha$ and $\beta$ in the CSCS method, to choose  $\alpha=\beta=\gamma^\star/2$, where $\gamma^\star$ is computed using \eqref{gammaStar}, seems to be a natural choice in the case that $A$ and $B$ have approximate norms. In practice Theorem \ref{theo:gammaStar} gives an efficent procedure to compute $\alpha$ and $\beta$ since we have the explicit formulae for the eigenvalues of the matrices $C_A$, $C_B$, $S_A$ and $S_B$ (see \eqref{explicitEigFormlae}) and we use these formulae to implement CSCS. Thus we can obtain the eigenvalues of $\widetilde{C}$ and $\widetilde{S}$ as a byproduct and verify if the sufficient condition for convergence given by Theorem \ref{CSCS convergence 1} is satisfied. See Algorithm \ref{alg:algorithm7} in Appendix \ref{ApendiceA}. 
Notice that the case $\theta_{min}=0$ brings no difficulty in computing $\gamma^{\star}$ -  
when $\theta_{min}=0$, we have  
$\tilde{\eta}=0$ and $\gamma^{\star}=\eta_{max}$.

\section{\normalsize\bf Numerical results}

 In this section we illustrate the performance of the CSCS algorithm exhibiting some  
numerical examples.  We compare the computational behavior of this method with the Hermitian and skew-Hermitian splitting iteration (HSS) \cite{Bai2011} and with a block  variant of 
the Symmetric Successive Over-Relaxation scheme (BSSOR) \cite{Kadry2007,Starke1991, DavidYoung1, DavidYoung}.

 All the algorithms were implemented in \textsc{Matlab}  (R2020b) in double precision (unit roundoff $\varepsilon=2.2\, 10^{-16}$) on a LAPTOP-KVSVAUU8 with an Intel(R) Core(TM) i5-8250U CPU @ 1.60GHz and 8 GB RAM, under Windows 10 Home. See Appendix \ref{ApendiceA} for details on the \textsc{Matlab} implementations (Algorithms \ref{alg:algorithm1}, \ref{alg:algorithm4} and \ref{alg:algorithm5} for CSCS, HSS and BSSOR, respectively). No parallel \textsc{Matlab} operations were used. 
 
 The built-in functions \texttt{fft} and \texttt{ifft} (Discrete Fourier transform and its inverse) were used in CSCS, in particular to compute the residual $R=C-AX-XB$ (see Algorithm \ref{alg:algorithm3} in Appendix \ref{ApendiceA}).
 The use of sparse techniques is an alternative way to compute the residual $C-AX-XB$. Indeed, if our matrices are stored in sparse format (even if only $A$, $B$ and $C$), then \textsc{Matlab} will
 automatically use highly efficient multiplication. The advantage of using Discrete Fourier transforms over these sparse techniques can only be observed for dense Toeplitz matrices (matrices with a \text{low} sparsity pattern or full matrices). See Example \ref{ExampleFull}.
 
 Hermitian and skew-Hermitian matrices can be diagonalizable by unitary matrices and thus it is posssible to treat the two steps at each iteration of the HSS method very efficiently - the linear systems are all diagonal. The diagonalization process is carried out by \textsc{Matlab} functions \texttt{schur} and \texttt{rsf2csf} for the real and complex Schur decompositions.

 We use a variant of the BSOR (block SOR) which combines two BSOR steps together in one iteration. Specifically, BSSOR is a forward BSOR step followed by a backward BSOR step. The roles of the triangular factors $L$ and $U$ of both $A$ and $B$ are reversed in the second step. The value of the relaxation parameter $\omega$ is the same in both steps. We remark here that 
 the application of SSOR (Symmetric SOR) as a preconditioner for other iterative schemes, in the case of symmetric matrices, was the primary motivation for SSOR, since the convergence rate is usually slightly slower than the convergence rate of SOR with optimal $\omega$. In our comparison study, in particular of the number of iterations needed for convergence, it seems more appropriate to use BSSOR than BSOR given that each iteration of BSSOR consists of two steps, like  CSCS and HSS. 
 
 The occurring linear systems in BSSOR are solved with the \textsc{Matlab} function \texttt{linsolve} 
 which uses $LU$ factorization with partial pivoting when the coefficient matrix is square. This function is more efficient than the backslash operator since it is possible to specify the appropriate solver as determined by the properties of the matrix.
 
  We also compare our method with the Bartels–Stewart direct method as implemented in the \textsc{Matlab} function \texttt{lyap} from the
 Control Toolbox. This function performs the real Schur decompositions of $A$ and $B$ in equation \eqref{original-equation}, lower and upper, respectively, and converts them afterwards to their complex forms; computes the solution of the resulting sylvester equation solving $m$ triangular systems and then transforms this solution back to the solution of the original Sylvester equation.
 See Algorithm \ref{alg:algorithm6} in Appendix \ref{ApendiceA} for our own implementation of this method (\texttt{mylyap} function).

The null matrix was chosen as the initial approximation, $X^{(0)}=O$, in all our numerical experiments, and the stopping criterion implemented was 
\begin{equation}\label{residtol}
\frac{ {\parallel R^{(k)} \parallel}_{F}}{ {\parallel C \parallel}_{F}} \leq tol,
\end{equation}
where $R^{(k)} = C-A X^{(k)} - X^{(k)} B $ is the residual attained at iteration $k$ and $tol$ is the desired accuracy, usually set to $10^{-6}$.

In our first example we analyze a standard Sylvester equation that comes from a finite difference discretization of the two dimensional convection-diffusion equation 
\begin{equation}\label{ConvectionDiffusion}
-(u_{xx}+u_{yy})+\sigma(x,y) u_x+\tau(x,y) u_y=f(x,y),
\end{equation}
posed on the  unit square $(0,1)\times(0,1)$ with Dirichlet-type boundary conditions. Here we consider the case when the coefficients $\sigma$ and $\tau$, which represent the velocity components along the $x$ and $y$
directions, respectively, are constant. See \cite[p. 371]{Chen2000}. A five-point discretization of the operator leads to a linear system 
\begin{equation}\label{equationConDiff}
\cal{A}\boldsymbol{u}=\boldsymbol{v},
\end{equation}
where now $\boldsymbol{u}$ denotes a vector in a finite-dimensional space. We consider a uniform $n\times n$ grid and use standard second-order finite differences for the Laplacian $u_{xx}+u_{yy}$ and either centered or upwind differences for the first derivatives $u_x$ and $u_y$. See \cite[p. 217]{ElmanGolub1991}. With $\boldsymbol{u}$ ordered lexicographically in the natural ordering as $(u_{11},u_{21}, ..,u_{nn})^T$, the coefficient matrix $\cal{A}$ is a block tridiagonal matrix whose $j$th row contains the subdiagonal, diagonal and superdiagonal blocks, all of order $n$, respectively,
\begin{equation}\label{exampleCDA}
{\cal{A}}_{j,j-1}=bI_n, \quad {\cal{A}}_{j,j}=\tridiag(c,a,d), \quad {\cal{A}}_{j,j+1}=eI_n,
\end{equation}
where $a,b,c,d$ and $e$ depend on the discretization. Blocks ${\cal{A}}_{1,0}$ and ${\cal{A}}_{n,n+1}$ are not defined.
Let $h = \dfrac 1{n + 1}$ ($n$ inner grid points in each direction). After scaling by $h^2$, the matrix entries are given by 
\begin{equation}\label{centeredScheme}
a=4, \quad b=-\left(1+\frac{\tau h}2\right), \quad c=-\left(1+\frac{\sigma h}2\right), 
\quad d=-\left(1-\frac{\sigma h}2\right), \quad e=-\left(1-\frac{\tau h}2\right),
\end{equation}
for the centered difference scheme, and by 
\begin{equation}\label{upwindScheme}
a=4+(\tau+\sigma)h, \quad b=-\left(1+\tau h\right), \quad c=-\left(1+\sigma h\right), 
\quad d=-1, \quad e=-1
\end{equation}
for the upwind scheme when $\sigma\ge0$ and $\tau \ge0$. At the $(i, j)$ grid point, the right-hand side satisfies $v_{ij}=h^2f_{ij}$, where $f_{ij}=f(ih,jh)$.

When $e=d$ and $b=c$, the coefficient matrix $\cal{A}$ in the linear system \eqref{equationConDiff} can be written in the form  
$
{\cal{A}}=I_{n} \otimes A + A \otimes I_{n}
$
where $A=\tridiag(c,a/2,d)$. Therefore, the Sylvester equation 
\begin{equation}\label{2DconvEq}
AX+XA^T=V
\end{equation}
is equivalent to the linear system \eqref{equationConDiff}, where $X$ and $V$ are the matrix-stacking of the vectors 
$\boldsymbol{u}$ and $\boldsymbol{v}$, respectively.

Different discretization schemes of equation  \eqref{ConvectionDiffusion} will naturally lead to different Sylvester equations (and different discretization errors). In \cite{Starke1991} it is described how we can obtain a general equation $AX+XB=C$ for any values of $\sigma$ and $\tau$ applying the central differences operator. Matrix $A$ corresponds to the discretization in the $y$-direction and matrix $B$ in the $x$-direction. 
When $\sigma$ and $\tau$ are constant, $A$ and $B$ are tridiagonal Toeplitz matrices defined by
\begin{gather}\label{CDscheme2}	
A=\tridiag\left(-1+\frac{\tau h}2,2,-1-\frac{\tau h}2\right) \text{ and }
B=\tridiag\left(-1+\frac{\sigma h}2,2,-1-\frac{\sigma h}2\right)
\end{gather}
with $A=B$ if $\tau=\sigma$.

\begin{example}\label{Example 01}
	 
Here we solve the Sylvester equation \eqref{2DconvEq} representing the convection-diffusion equation \eqref{ConvectionDiffusion} with 
homogeneous Dirichlet boundary conditions and the function $f$ defined by
$
f(x,y)=\boldsymbol{e}^{x+y}.
$
Different values for $\tau=\sigma$ and the step size $h=\frac{1}{n+1}$ are considered. 

\end{example}

The performance of all the methods, BSSOR, HSS and CSCS, concerning the number of iterations (iter) and CPU time in seconds ($\text{t}_{\text{CPU}}$) are shown in Tables \ref{TableExample0}, 
 for the centered differences scheme \eqref{centeredScheme}. The results for the upwind scheme \eqref{upwindScheme} and for the alternative scheme \eqref{CDscheme2} are pratically the same. 

In the CSCS method we took  $\alpha=\beta\approx\gamma^\star$, where $\gamma^\star$ is computed using the expression \eqref{gammaStar}, and for the HSS method we chose $\alpha=\beta\approx2\gamma^\star$ (as a result of a numerical search around $\gamma^\star/2$); for the relaxation parameter $\omega$ in the BSSOR method we used the heurist estimate given by $\omega=2-10h$ (approximately).

 We report that the initial matrix $C_A$ is positive semi-definite but $S_A$ is positive definite 
 (as well as $\widetilde{C}=I_n\otimes C_{A}+C_{A}\otimes I_n$ and $\widetilde{S}=I_n\otimes S_{A}+S_{A}\otimes I_n$, respectively), and thus the CSCS method always converges. In fact, we can prove that this splitting property of the matrix 
 $A=\tridiag(c,a/2,d)$ is true in general for any positive values of $\sigma$ and $\tau$.

	For $h=0.05$ the BSSOR method converged but very slowly. It took more than 20 minutes to deliver a solution, with relative residual norm of about $10^{-4}$ (1500 iterations), for  $\sigma=2$, and  $10^{-6}$ (946 iterations), for $\sigma=10$. It is not a suitable method for this case.
	
	Overall, the number of iterations needed for convergence by all the methods is relatively high and this reflects the fact that the spectral radi of the iteration matrices are closer to $1$ than to $0$. Nevertheless, CSCS exhibits the best behavior among the three methods. Our method is
	nealy 3 times faster than HSS (5 times for $\sigma=2$, $n=399$) and $8$ times faster (in average) than BSSOR, for $n\leq 199$ (much faster for $n>199$). 
	
\pagebreak 

\renewcommand{\arraystretch}{1.05}
\begin{table}[h!]
	\small
	\centering
	\begin{tabular}{c l |c c c |c c c|c c c}\toprule
		\multicolumn{2}{c|}{}  & \multicolumn{3}{c}{BSSOR}  &  \multicolumn{3}{c}{HSS} &\multicolumn{3}{c}{CSCS} 
		\\ \cline{3-11}
		\multicolumn{1}{c}{} & \multicolumn{1}{l|}{$h=\frac1{n+1}$} &
		\multicolumn{1}{c}{$\omega$}
		&\multicolumn{1}{c}{iter} &\multicolumn{1}{c|}{ $\text{t}_{\text{CPU}}$} & 
		\multicolumn{1}{c}{$\alpha=\beta$} & \multicolumn{1}{c}{iter} &\multicolumn{1}{c|}{$\text{t}_{\text{CPU}}$}&
		\multicolumn{1}{c}{$\alpha=\beta$} & \multicolumn{1}{c}{iter} &\multicolumn{1}{c}{$\text{t}_{\text{CPU}}$}
		\\ \midrule
		& $0.04$ & 1.75  &79  & 0.04  & 0.20    & 85   &0.01   & 0.10  & 42  & 0.005\\
		& $0.02$ & 1.85  &167  & 0.25  & 0.10   & 167   &0.07   & 0.045 & 84  & 0.03\\
		$\sigma=2$ & $0.01$ & 1.95 &309  & 2.25  & 0.050  &328  &0.62   & 0.023 & 168 & 0.25\\
		& $0.005$ & 1.95 &767 & 51.6  &0.025 &648  &5.41   & 0.011 & 342 & 1.90\\
		& $0.0025$   &  - &- & -  &0.013 & 1285  &90.3  & 0.006 & 700 & 18.2\\\midrule
		& $0.04$ & 1.75  & 36 & 0.02 & 0.45    & 64   &0.01  & 0.20  & 29  & 0.006\\
		& $0.02$ & 1.85  & 69  & 0.11  & 0.22    & 126   &0.06   & 0.075 & 56  & 0.02\\
		$\sigma=10$ & $0.01$   & 1.85 &190 & 1.44  & 0.11  & 252   & 0.44  & 0.038 & 108 & 0.17\\
		& $0.005$   & 1.95 &258  & 22.6  & 0.05 & 448  &3.38   & 0.019 & 216 & 1.22\\
		& $0.0025$  & - &- & -  & 0.013 & 841  & 66.0  & 0.0094 & 438 & 20.9\\
		\bottomrule
	\end{tabular}
	\caption{BSSOR, HSS and CSCS performance for the Example \ref{Example 01} (centered difference scheme).}
	\label{TableExample0}
\end{table}

In \cite{Benner2014, Jbilou2013, Jbilou2006} the authors study the numerical solution of \eqref{ConvectionDiffusion} with non-constant coefficients. The discretization matrices ${\cal{A}}_1$ and ${\cal{A}}_2$ from two different linear systems \eqref{equationConDiff} are used to create a Sylvester equation
\begin{equation}\label{TestPurpose}
{\cal{A}}_1X+X{\cal{A}}_2={\cal{C}},
\end{equation}
where ${\cal{C}}$ is randomly generated from values uniformly distributed in $[0, 1]$.  These numerical examples were devised entirely for testing purposes and they are not connected to the solution of \eqref{equationConDiff}. We will imitate this type of examples but in our case matrices ${\cal{A}}_1$ and ${\cal{A}}_2$ must be Toeplitz.

\begin{example}\label{Example 02}
	
We slightly change matrix ${\cal{A}}$, defined by \eqref{exampleCDA}, to have constant diagonal, subdiagonal, superdiagonal, $n^\text{th}$ diagonal and $-n^\text{th}$ diagonal (values $a$, $c$, $d$, $e$ and $b$, respectively). 
For diferente values of $\sigma=\tau$, we define ${\cal{A}}_1$ and ${\cal{A}}_2$, with orders $n^2$ and $m^2$, respectively, and solve \eqref{TestPurpose}. 
Table \ref{TableExample01} shows the outcome of this experiment. 

\end{example}

\renewcommand{\arraystretch}{1.05}
\begin{table}[h!]
	\small
	\centering
	\begin{tabular}{c l |c c c |c c c|c c c}\toprule
		\multicolumn{2}{c|}{}  & \multicolumn{3}{c}{BSSOR}  &  \multicolumn{3}{c}{HSS} &\multicolumn{3}{c}{CSCS} 
		\\ \cline{3-11}
		\multicolumn{1}{c}{} & \multicolumn{1}{l|}{$n^2;m^2$} &
		\multicolumn{1}{c}{$\omega$}
		&\multicolumn{1}{c}{iter} &\multicolumn{1}{c|}{ $\text{t}_{\text{CPU}}$} & 
		\multicolumn{1}{c}{$\alpha=\beta$} & \multicolumn{1}{c}{iter} &\multicolumn{1}{c|}{$\text{t}_{\text{CPU}}$}&
		\multicolumn{1}{c}{$\alpha=\beta$} & \multicolumn{1}{c}{iter} &\multicolumn{1}{c}{$\text{t}_{\text{CPU}}$}
		\\ \midrule
		& $49;100$ & 1.75  &38  & 0.12  & 0.89   & 49   &0.05   & 0.60   & 30   &0.02\\
		& $100;100$  &1.75  & 39  & 0.28   & 0.81   &65   &0.13 & 0.41  & 33 & 0.05\\
		$\sigma_1=\sigma_2=2$ & $225;225$ &1.75 &63  & 4.84 & 0.45  & 92 &1.20  &0.27 &46  & 0.25\\
		& $225;400$   & 1.85  &67  &24.5  & 0.42   & 99   &3.92   &0.28 &58  &0.76\\
		& $400;400$   & 1.85 &74 & 61.7  & 0.35   & 117   &9.20  &0.20 & 61 &1.18\\
		& $625;625$   & 1.75 &98 & 333.7  & 0.29   & 143   &39.1  &0.19 & 89 &4.63\\\midrule
		&  $100;100$  & 1.75 &39  &0.28 & 0.87   & 69 &0.15  &0.87   & 69 & 0.13\\
		$\sigma_1=1; \sigma_2=10$ & $225;400$ &1.85 &67  &27.89 & 0.45   & 100 &4.03  &0.31   & 64 & 0.88\\
		& $625;625$   & $1.85$ &$98$ & $335.0$  &$0.27$ &$150$ &40.7 &0.14 &74  &4.00\\
		& $784;784$   & $1.85$ &$116$ & $700.3$  &$0.29$ &$164$ &89.5 &0.15 &84  &7.00\\
		\bottomrule
	\end{tabular}
	\caption{Performance of BSSOR, HSS and CSCS for the Example \ref{Example 02}.}
	\label{TableExample01}
\end{table}

 We verified that the matrices $\widetilde{C}$ and $\widetilde{S}$ are positive semi-definite and positive definite, respectively, like in the first example. 
	The values for the parameters $\alpha$ and $\beta$ that led to a smaller number of iterations were values greater than the value $\gamma^\star$ given by \eqref{gammaStar}, by a factor of about $10$ or higher. 
	
	The convergence rate of all the methods is faster for this example than for the previous one and, as expected, BSSOR is a very slow method compared to HSS and CSCS. Also in this case, when compared to HSS, the CSCS method is about 5 times faster, for matrices of order $200$, and 8 times faster if the order of the matrices raises above $600$. 

The next numerical example can be found, for instance, in \cite{Bai2011, Wang2013, Zheng2014}. As mentioned in \cite{Bai2011} this class of problems appears associated with the preconditioned Krylov subspace iteration method used to solve the systems of linear equations which arise from the discretization of various differential equations and boundary value problems using finite difference or Sinc-Galerkin schemes.

\begin{example}\label{Example 1}
	Consider the Sylvester equation {\rm (\ref{original-equation})} with matrices $A,B\in\mathbb{C}^{n\times n}$ $(m=n)$ defined by
	$$A=B=M+2rN+\frac{100}{(n+1)^{2}}I,$$
	where $M,N \in\mathbb{C}^{n\times n}$ are Toeplitz tridiagonal matrices, 
	$M=\tridiag(-1,2,-1)$, $N=\tridiag(0.5, 0, -0.5)$. In a more compacted form,
	\[
	A=B=\tridiag\left(-1+r,2+\frac{100}{(n+1)^{2}},-1-r\right).
	\]
	The parameter $r$ depends on the properties of the problem being discretized. 
\end{example} 

Although this problem is similar to the one considered in Example \ref{Example 01}, we decided to show the results of our experiments in order to compare them with the results presented by other authors, namely in \cite{Bai2011, liu2020, Wang2013}.
Table \ref{TableExample1} contains the summary of our experiments for different instances of the parameter $r$ and the 
order $n$ of the matrices.

\renewcommand{\arraystretch}{1.05}
\begin{table}[h!]
	\small
	\centering
	\begin{tabular}{c c |c c c |c c c|c c c}\toprule
		\multicolumn{2}{c|}{}  & \multicolumn{3}{c}{BSSOR}  &  \multicolumn{3}{c}{HSS} &\multicolumn{3}{c}{CSCS} \\ \cline{3-11}
		\multicolumn{1}{c}{} & \multicolumn{1}{c|}{$n$} &
		\multicolumn{1}{c}{$\omega$}
		&\multicolumn{1}{c}{iter} &\multicolumn{1}{c|}{ $\text{t}_{\text{CPU}}$} & 
		\multicolumn{1}{c}{$\alpha$} & \multicolumn{1}{c}{iter} &\multicolumn{1}{c|}{$\text{t}_{\text{CPU}}$}&
		\multicolumn{1}{c}{$\alpha=\beta$} & \multicolumn{1}{c}{iter} &\multicolumn{1}{c}{$\text{t}_{\text{CPU}}$}\\ \midrule
		& $64$   & 1.75 &36  & 0.10 & 0.17 & 123  &0.10   & 0.130 & 32 & 0.02\\
		& $128$  & 1.85 &71  & 1.08 & 0.09 & 244  &0.75  & 0.070 & 60 & 0.14\\
$r=0.01$& $256$  & 1.95 &167 & 24.0  & 0.05 & 453  &19.2 & 0.035 & 112 & 0.88\\
        & $512$  & 1.95 &281 & 670.3 & 0.05  & 520  & 64.0 & 0.017 & 221 & 7.00\\
        & $1024$ & - &- & - & 0.01  & 1204 & 1408& 0.010 & 392 & 61.0\\\midrule
		& $64$   & 1.75 &35  & 0.12  & 0.23 & 90  &0.07   & 0.14 & 31 & 0.02\\
		& $128$  & 1.85 &65  & 1.10  & 0.13 & 145  &0.46  & 0.08 & 55 & 0.13\\
$r=0.1$ & $256$  & 1.85 &139 & 21.4 & 0.09 & 219  &4.26 & 0.05 & 86 & 0.70\\
        & $512$  & 1.75 &455 & 678.1& 0.10  & 314  & 33.4 & 0.10 & 317 & 10.4\\
        & $1024$ & - &- & - & 0.10  & 607 & 636.4& 0.10 & 610 & 100.5\\\midrule
		& $64$   & 1.5  &22  & 0.06  & 0.81 & 40  &0.04 &0.26 & 26 & 0.01\\
		& $128$  & 1.5  &33  & 0.50  & 0.62 & 60  &0.19  & 0.16 & 41 & 0.09\\
  $r=1$ & $256$  & 1.75 &47  &7.10 & 0.51 & 92  &1.81 & 0.11 & 61 & 0.45\\
        & $512$  & 1.75 &62 & 114.2& 0.25  & 132  & 16.7 & 0.25 & 138 & 3.95\\
        & $1024$ & - &- & - & 0.25 & 192 & 173.0 &0.15 & 171 & 21.3\\
		\bottomrule
	\end{tabular}
	\caption{Performance of BSSOR, HSS and CSCS methods for the Example \ref{Example 1}.}
	\label{TableExample1}
\end{table}

\pagebreak
The values of the shift parameter $\alpha$ in the HSS method are the values which were presented in \cite{Bai2011}, for $n\leq256$ (see $\omega_{\exp}$ and $\alpha_{\exp}$ in \cite[Table 4.2]{Bai2011}, obtained through an experimental search).
The values given to the shift parameters $\alpha$ and $\beta$ in the CSCS method were determined using the expression \eqref{gammaStar} - we computed $\gamma^\star$ and let $\alpha=\beta$ between $\gamma^\star/8$ and $\gamma^\star/2$ - and these values are also used with HSS when $n>256$.

In this example the convergence is faster than in Example \ref{Example 01}, in particular when $r=1$. Matrices $\widetilde{C}$ and $\widetilde{S}$ are both positive definite and the CSCS method outperforms the HSS and BSSOR methods both in terms of the number of iterations and in what respects to the computational efficiency. BSSOR may be very slow for matrices of order $n\ge512$,  taking more than $20$ minutes to converge. Compared to HSS the CPU time required by CSCS to converge is, in most cases, $4$ to $8$ times smaller ( in extreme cases, this factor may be much smaller). Except for $r=1$, our implementation of HSS demands a higher number of iterations than shown in \cite{Bai2011} for this same method, but despite this, in all cases the CPU time needed is reduced.



 
The advantage of using FFT operations in the CSCS method can be entirely appreciated when we take $A$ and $B$ to be full Toeplitz matrices. Next example considers this case and reports the CPU elapsed times for CSCS and \textsc{Matlab} function {\texttt{lyap}}.
 
\begin{example}\label{ExampleFull}
This example  takes positive definite circulant and skew-circulant matrices $C_A$ and $S_A$ (obtained using translation of origin on randomly generated matrices) and forms $A=C_A+S_A$, $B=A$. Matrix $C$ is chosen to be the matrix attained when all the entries in $X$ are set to be $1$. 

We take $\alpha=\beta=\gamma^\star/2$ where $\gamma^\star$ is computed using the expression \eqref{gammaStar} in Theorem \ref{theo:gammaStar}. See Table \ref{TableExampleFull} for a comparison of the efficiency of CSCS and {\texttt{lyap}}. 

\renewcommand{\arraystretch}{1.05}
\begin{table}[h!]
	\small
	\centering
	\begin{tabular}{c |c c c c |c c }\toprule
		\multirow{2}{*}{$n$}  & \multicolumn{4}{c}{CSCS}  &
		\multicolumn{2}{c}{\texttt{lyap}}\\ \cline{2-7}
	    &\multicolumn{1}{c}{$\alpha=\beta$}
		&\multicolumn{1}{c}{iter} &\multicolumn{1}{c}{ resid} &\multicolumn{1}{c|}{ $\text{t}_{\text{CPU}}$} & 
		\multicolumn{1}{c}{resid} & \multicolumn{1}{c}{$\text{t}_{\text{CPU}}$}\\ \midrule
		\multirow{2}{*}{$100$}& \multirow{2}{*}{$43.49$} &5&$1.1\,\;10^{-6}$&$0.017$&\multirow{2}{*}{$2.2\,\;10^{-15}$}&\multirow{2}{*}{$0.013$}\\
		& &12&$1.9\,\;10^{-15}$&0.040&\\
		\multirow{2}{*}{$250$}& \multirow{2}{*}{$103.75$} &5&$2.0.4\,\;10^{-6}$&$0.07$&\multirow{2}{*}{$1.8\,\;10^{-15}$}&\multirow{2}{*}{$0.10$}\\
		& &13&$1.5\,\;10^{-15}$&0.14&\\
		\multirow{2}{*}{$500$}& \multirow{2}{*}{$208.0$} &5&$2.0\,\;10^{-6}$&$0.32$&\multirow{2}{*}{$1.9\,\;10^{-15}$}&\multirow{2}{*}{$0.26$}\\
		& &12&$4.0\,\;10^{-15}$&0.67&\\
		\multirow{2}{*}{$1000$}& \multirow{2}{*}{$426.6$} &5&$1.3\,\;10^{-6}$&$1.42$&\multirow{2}{*}{$2.1\,\;10^{-15}$}&\multirow{2}{*}{$1.27$}\\
		& &12&$9.4\,\;10^{-15}$&3.06&\\
		\multirow{2}{*}{$1500$}& \multirow{2}{*}{$645.4$} &5&$1.2\,\;10^{-6}$&$3.20$&\multirow{2}{*}{$2.3\,\;10^{-15}$}&\multirow{2}{*}{$3.34$}\\
		& &13&$3.7\,\;10^{-16}$&7.43&\\
		\multirow{2}{*}{$2000$}& \multirow{2}{*}{$856.97$} &5&$1.7\,\;10^{-6}$&$5.62$&\multirow{2}{*}{$2.5\,\;10^{-15}$}&\multirow{2}{*}{$7.01$}\\
		& &12&$8.2\,\;10^{-15}$&12.01&\\
		\multirow{2}{*}{$2500$}& \multirow{2}{*}{$1080.1$} &5&$1.7\,\;10^{-6}$&$10.64$&\multirow{2}{*}{$2.7\,\;10^{-15}$}&\multirow{2}{*}{$16.01$}\\
		& &13&$1.0\,\;10^{-15}$&23.78&\\
		\bottomrule
	\end{tabular}
	\caption{Performance of CSCS and  {\texttt{lyap}} for full matrices $A$ and $B$.}
	\label{TableExampleFull}
\end{table}
\end{example}

If the relative accuracy demanded is ${{\cal{O}}(10^{-6})}$, which is often enough in many applications, the CSCS method is comparable to or even faster than \texttt{lyap}. When full accuracy ${{\cal{O}}(\varepsilon)}$ is important, more iterations are needed and CSCS takes approximately twice as long as \texttt{lyap}, which, however, can still be considered very satisfactory since these methods are fast even for large dimensions like $n \ge 1000$. 

We may take our function \texttt{mylyap} (see Algorithm \ref{alg:algorithm6}) in this comparison study, which is possibly the fairest comparison study to present, given that in our implementations we are not capable of reproducing the \textsc{Matlab} internal linear systems solvers used by \texttt{lyap}. We clearly aknowledge that, when full accuracy ${{\cal{O}}(\varepsilon)}$ is required, CSCS method is always faster, about 10 times faster, than \texttt{mylyap} for full Toeplitz matrices $A$ and $B$ (\texttt{mylyap} is, as expected, slower than \texttt{lyap}).

\section{Conclusions} We considered the problem of solving a large continuous Sylvester equation\break $ AX + XB = C$ where the coefficient matrices $A$ and $B$ are assumed to be Toeplitz matrices and we have devised the CSCS iteration which is a method based on the circulant and skew-circulant splittings of the matrices $A$ and $B$. The spectral properties of these structured matrices allow the use of fast Fourier transforms (FFTs) which reduces significantly the operation count of matrix multiplication and thus the  computational efficiency of the algorithm. We have also analyzed sufficient conditions for the convergence of the CSCS iteration and have derived an upper bound for its convergence factor. The numerical experiments we have carried out illustrate that CSCS is a faster and more robust iterative algorithm than the alternatives HSS and BSSOR. The advantage of using FFT operations in the CSCS method can be entirely appreciated when we take $A$ and $B$ to be full Toeplitz matrices and in this case CSCS is a very competitive algorithm even when compared with the \textsc{Matlab} function \texttt{lyap} which implements the Bartels–Stewart direct method.
Moreove, since FFT-based operations have very high parallel potentialities, our CSCS algorithm is therefore suited for parallel frameworks.

\newpage
\appendix
\section{Implementation details}
\label{ApendiceA}
\begin{algorithm}
	\caption{\textbf{\; CSCS -- circulant and skew-circulant splitting iteration}}
	\label{alg:algorithm1}
	\begin{algorithmic}
		\Statex
		
		\hspace*{-0.4cm}\textbf{Input:} Toeplitz matrices $A, B$, $C$ (orders $n\times n$, $m\times m$ and $n\times m$),
		
		\hspace*{0.1cm} initial approximation $X_0$, relative residual tolerance $tol$, maximum number of iterations $maxit$ 
		\Statex
		\hspace*{-0.4cm}\textbf{Output:} Solution $X$ of the Sylester equation $AX+XB=C$
		
		\Statex
		
		\State [$C_A$,\,$S_A$] = \texttt{CSsplitting}($Acol_1$,$Arow_1$) \Comment{ circulant and skew-circulant splittings of $A$ and $B$}
		\State [$C_B$,\,$S_B$] = \texttt{CSsplitting}($Bcol_1$,$Brow_1$)
		\Statex	
		\State $D_n=\texttt{exp}((0:n-1)/n*\texttt{pi}*\textbf{i})$ \Comment{ $D_n=[1,\operatorname{\boldsymbol{e}}^{\frac{\pi }{n}\textbf{i}},
		\ldots,\operatorname{\boldsymbol{e}}^{\frac{(n-1)\pi}{n}\textbf{i}}]$}
		\State $D_m=\texttt{exp}((0:m-1)/m*\texttt{pi}*\textbf{i})$
		
		
	    \State $D_{cA}=\texttt{diag}\big(\texttt{ifft}\big(\texttt{fft}(C_A).\texttt{'}\big).\texttt{'}\big)$       
	    \Comment{ $D_{cA}=\texttt{diag}(\Lambda_A); \, \Lambda_A=F_AC_AF_A^*$}  
	    	
		\State $D_{cB}=\texttt{diag}\big(\texttt{ifft}\big(\texttt{fft}(C_B).\texttt{'}\big).\texttt{'}\big)$           
		\Comment{ $D_{cB}=\texttt{diag}(\Lambda_B);\, \Lambda_B=F_BC_BF_B^*$}

		\State $D_{sA}=\texttt{diag}(\texttt{fft}(\texttt{ifft}(D_n\texttt{'}\bs{.*}S_A\bs{.*}D_n).\texttt{'}).\texttt{'})$ 
		\Comment{ $D_{sA}=\texttt{diag}(\Sigma_A); \, \Sigma_A=\hat{F}_AS_A\hat{F}_A^*$} 
		\State $D_{sB}=\texttt{diag}(\texttt{fft}(\texttt{ifft}(D_m\texttt{'}\bs{.*}S_B\bs{.*}D_m).\texttt{'}).\texttt{'})$ 
		\Comment{ $D_{sB}=\texttt{diag}(\Sigma_B); \, \Sigma_B=\hat{F}_BS_B\hat{F}_B^*$}  
		
		\State \Comment{ $\bs{.*}$ for element-wise product}
		
		\State [$\alpha$,\,$\beta$] = \texttt{shifts}($D_{cA},D_{cB},D_{sA},D_{sB}$) \Comment{shift parameters $\alpha$ and $\beta$}	
		\State $D_1=D_{C_A}+\alpha$; \, $D_2=D_{C_B}+\beta$; \, $D_3=D_{S_A}+\alpha$;\, $D_4=D_{S_B}+\beta$
		
		\Statex	
		\State  $A=\texttt{sparse}(A);\; B=\texttt{sparse}(B); \; C=\texttt{sparse}(C)$ \Comment{ \textsc{Matlab} sparse matrix storage format}

		\State $X=X_0$
		
		\State $R=\texttt{resid}(X,C,D_{cA},D_{cB},D_{sA},D_{sB},D_n,D_m)$ \Comment{  $R=C-AX-XB$ using \eqref{residualGeneral}}
		\Statex
		\State $normR=\texttt{norm}(R,\texttt{`fro'})$; \;  
		$normC=\texttt{norm}(C,\texttt{`fro'})$ \Comment{ Frobenious norms of $R$ and $C$}
	    
		\State $\texttt{iter}=0$ \Comment{ number of iterations counter}
		
		\While {\big($(normR/normC)>tol$ \, and \, $\texttt{iter}<maxit$\big)}
		\State \% First step
		\State $R=(\texttt{ifft}(\texttt{fft}(R).\texttt{'}).\texttt{'})*(\texttt{sqrt}(m/n))$ 
		\Comment{ transform \eqref{residual1}}
		
		\State $Z=R./(D1+D2.\texttt{'})$  \Comment{ solve  $(\alpha I+\Lambda_A)Z+Z(\beta I+\Lambda_B)=R$} 
		
		\State $X=X+\texttt{ifft}((\texttt{fft}(Z.\texttt{'})).\texttt{'})$ \Comment{ update $X$; $X\leftarrow X+{F_A^*}{Z}F_B$}
		
		\Statex
		\State \% Second step
		
		\State $R=\texttt{resid}(X,C,D_{cA},D_{cB},D_{sA},D_{sB},D_n,D_m)$
		\Comment{ use \eqref{residualGeneral} to compute R}
		
		
		\State $R=\texttt{ifft}((\texttt{fft}(((D_n\texttt{'}).*R.*D_m).\texttt{'})).\texttt{'})*(\texttt{sqrt}(n/m))$  
		\Comment{ transform \eqref{residual2}}
		
		\State $Z=R./(D_3+D_4.\texttt{'})$   
		\Comment{ solve $(\alpha I+\Sigma_A){Z}+{Z}(\beta I+\Sigma_B)=R$ }
		
		\State $X=X+(D_n.\texttt{'}).*(\texttt{ifft}(\texttt{fft}(Z).\texttt{'}).\texttt{'}).*\texttt{conj}(D_m)$
		\Comment{ update $X$; $X\leftarrow X+\hat F_A^*Z\hat F_B$}
		\Statex
			
	    \State $R=\texttt{resid}(X,C,D_{cA},D_{cB},D_{sA},D_{sB},D_n,D_m)$
		
		\State $normR=\texttt{norm}(R,\texttt{`fro'})$
		
		\State $\texttt{iter}=\texttt{iter}+1$
		
	    \EndWhile	
	    
	    
	    \If{ $\texttt{iter}>=\texttt{maxit}$}
	    \State \texttt{disp}(\texttt{`}Maximum number of iterations exceed.\texttt{'})
	    \EndIf
	
	\end{algorithmic}
\end{algorithm}

\begin{algorithm}
	\caption{\textbf{\; CSsplitting -- circulant and skew-circulant splitting of a Toeplitz matrix}}
	\label{alg:algorithm2}
	\begin{algorithmic}
		\Statex
		
		\hspace*{-0.4cm}\textbf{Input:} First column and first row, $c$ and $r$, of a $n\times n$ Toeplitz matrix $A$ 
		
		\hspace*{0.1cm} ($c$ and $r$ should be given as rows, $c(1)$ should be equal to $r(1)$)

		\Statex
		\hspace*{-0.4cm}\textbf{Output:} $C_A$ and $S_A$ with $A=C_A+S_A$, circulant and skew-circulant splitting of $A$
		
		\Statex
		
		\State \% Circulant part 
		\State $C_{Ac}=(c+[0,\texttt{fliplr}(r(2:n))])/2$
		\State $C_{Ar}=[r(1)/2, \texttt{fliplr}(C_{Ac}(2:n))]$
		\State $C_A=\texttt{toeplitz}(C_{Ac},C_{Ar})$
		
		\Statex
			
		\State \% Skew-circulant part
		\State $S_{Ac}=(c-[0,\texttt{fliplr}(r(2:n))])/2$   
		\State $S_{Ar}=[r(1)/2 -\texttt{fliplr}(S_{Ac}(2:n))]$
		\State $S_A=\texttt{toeplitz}(S_{Ac},S_{Ar})$		
	
	    \medskip
		
	\end{algorithmic}
\end{algorithm}

\bigskip
\bigskip

\begin{algorithm}
	\caption{\textbf{\; resid -- residual for a given approximation $X$}}
	\label{alg:algorithm3}
	\begin{algorithmic}
		\Statex
		
		\hspace*{-0.4cm}\textbf{Input:} Approximation $X$ and matrix $C$ of $AX+XB=C$, 
		
		\hspace*{0.1cm} $D_{cA}$, $D_{cB}$, $D_{sA}$, $D_{sB}$, $D_n$ and $D_m$, computed in \texttt{CSCS} function
	
		\Statex
		\hspace*{-0.4cm}\textbf{Output:} Residual $R=C-AX-XB$ using  \eqref{residualGeneral}
		
		\Statex

		\State $p_1=\texttt{ifft}(D_1\bs{.*}\texttt{fft}(X))$
		\Comment{ $p_1=F_A^*\Lambda_{A}F_AX$}

		\State $p_2=(D_n.\texttt{'})\bs{.*}(\texttt{fft}(D_3\bs{.*}\texttt{ifft}((D_n\texttt{'})\bs{.*}X)))$  
		\Comment{ $ p_2=\hat{F}_A^*\Sigma_A \hat{F}_AX $}

		\State $p_3=\texttt{fft}(((\texttt{ifft}(X.\texttt{'}).\texttt{'})\bs{.*}(D_2.\texttt{'})).\texttt{'}).\texttt{'}$
		\Comment{ $p_3=XF_B^*\Lambda_{B}F_B$}

		\State $p_4=(\texttt{fft}((X\bs{.*}D_m).\texttt{'})).\texttt{'}$                     
		\State $p_4=(\texttt{ifft}((p_4\bs{.*}(D_4.\texttt{'})).\texttt{'}).\texttt{'})\bs{.*}(\texttt{conj}(D_m))$   
		\Comment{ $p_4=X\hat{F}_B^*\Sigma_B\hat{F}_B$} 
		
		\State $R=C-p_1-p_2-p_3-p_4$
		
		\medskip
		
\end{algorithmic}				
\end{algorithm}

\begin{algorithm}
	\caption{\textbf{\; shifts -- find $\gamma^{\star}$ and shift parameters $\alpha$ and $\beta$}}
	\label{alg:algorithm7}
	\begin{algorithmic}
		\Statex
		
		\hspace*{-0.4cm}\textbf{Input:} $D_{cA},D_{cB},D_{sA}$, $D_{sB}$ eigenvalues of $C_A$, $C_B$, $S_A$, $S_B$, respectively
	
		\Statex
		\hspace*{-0.4cm}\textbf{Output:} shifts $\alpha$ and $\beta$
		
		\Statex
		
		\State $D_{\widetilde{C}}=D_{cA}+D_{cB}.\texttt{'}$; \, $D_{\widetilde{S}}=D_{sA}+D_{sB}.\texttt{'}$ \Comment{ eigenvalues of $\widetilde{C}$ and $\widetilde{S}$}

		\Statex
		\State $D_1=[\texttt{real}(D_{\widetilde{C}});\,\texttt{real}(D_{\widetilde{S}})]$; \; 
		       $D_2=\texttt{abs}\big([\texttt{imag}(D_{\widetilde{C}});\, \texttt{imag}(D_{\widetilde{S}})]\big)$
		       \Comment{ see Remark \ref{remarkTetaEta}}
		\State $\theta_{min}=
		\texttt{min}\big(\texttt{min}(D_1)\big)$;\; $\theta_{max}=
		    \texttt{max}\big(\texttt{max}(D_1)\big)$
		
	    \State $\eta_{min}=\texttt{min}\big(\texttt{min}(D_2)\big)$;\, $\eta_{max}=\texttt{max}\big(\texttt{max}(D_2)\big)$	
		
		\Statex
		
		\If{$\theta_{min}\ge0$} \Comment{ see Theorem \ref{theo:gammaStar}}
		\If{$\eta_{max}<\sqrt{\theta_{min}*(\theta_{max}-\theta_{min})/2}$}
		\State $\gamma^{\star}=\sqrt{\theta_{min}\theta_{max}-\eta_{max}^{2}}$
		\Else
		\State {$\gamma^{\star}=\sqrt{\theta_{min}^{2}+\eta_{max}^{2}}$}
		 \EndIf
		\Else
		\State $\gamma^{\star}=1$ \Comment{ random value for $\gamma^{\star}$ if $\theta_{min}<0$} 
		\State \texttt{disp}(\texttt{`}Warning: $\theta_{min}<0$. We let $\gamma^{\star}=1$.\texttt{'})
		\EndIf
		\State $\alpha=\gamma^{\star}/2$; \;  $\beta=\alpha$

\end{algorithmic}				
\end{algorithm}

\begin{algorithm}
	\caption{\textbf{\; HSS -- Hermitian and skew-Hermitian splitting iteration}}
	\label{alg:algorithm4}
	\begin{algorithmic}
		\Statex
		
		\hspace*{-0.4cm}\textbf{Input:} Toeplitz matrices $A, B$, $C$ (orders $n\times n$, $m\times m$ and $n\times m$),
		
		\hspace*{0.1cm} initial approximation $X_0$, shift parameters $\alpha$ and $\beta$, relative residual tolerance $tol$,
		
		\hspace*{0.1cm}  maximum number of iterations $maxit$ 
		\Statex
		\hspace*{-0.4cm}\textbf{Output:} Solution $X$ of the Sylester equation $AX+XB=C$
		
		\Statex
		\State\Comment{ Hermitian and skew-Hermitian splittings of $A$ and $B$}
		\State $H_1=(A+A\texttt{'})/2$ 
		\State $S_1=(A-A\texttt{'})/2$ 
		\State $H_2=(B+B\texttt{'})/2$
		\State $S_2=(B-B\texttt{'})/2$ 
		\State  \Comment{ schur forms of $H_1$, $H_2$, $S_1$ and $S_2$ (diagonalizable)}    
		 \State [$Q_1$,\,$D_1$] = \texttt{schur}(\texttt{full}($H_1$))    
		                         \Comment{ $H_1=Q_1D_1Q_1^*$}             
		 
		 \State [$Q_2$,\,$D_2$] = \texttt{schur}(\texttt{full}($H_2$)) \Comment{ $H_2=Q_2D_2Q_2^*$}
		 
		 \Statex

	     \State [$Q_3$,\,$D_3$] = \texttt{schur}(\texttt{full}($S_1$))
		 \State [$Q_3$,\,$D_3$] = \texttt{rsf2csf}($Q_3$,$D_3$)  \Comment{ $S_1=Q_3D_3Q_3^*$}
		
		 \State [$Q_4$,\,$D_4$] = \texttt{schur}(\texttt{full}($S_2$))
		 \State [$Q_4$,\,$D_4$] = \texttt{rsf2csf}($Q_4$,$D_4$) \Comment{ $S_2=Q_4D_4Q_4^*$}
		 
		\Statex
		
		\State  \Comment{ diagonal elements of $D_1+\alpha I_n$, $D_2+\beta I_m$, $D_3+\alpha I_n$ and $D_4+\beta I_m$}
		
		\State $D_1=\texttt{diag}(D_1)+\alpha$; \, $D_2=\texttt{diag}(D_2)+\beta$ 
		\State $D_3=\texttt{diag}(D_3)+\alpha$;\, $D_4=\texttt{diag}(D_4)+\beta$

		\Statex	
		\State  $A=\texttt{sparse}(A);\; B=\texttt{sparse}(B); \; C=\texttt{sparse}(C)$ \Comment{ \textsc{Matlab} sparse matrix storage format}

		\State $X=X_0$
		
		\State $R=C-A*X-X*B$ 
		\Statex
		\State $normR=\texttt{norm}(R,\texttt{`fro'})$; \;  
		$normC=\texttt{norm}(C,\texttt{`fro'})$ \Comment{ Frobenious norms of $R$ and $C$}
		
		\State $\texttt{iter}=0$ \Comment{ number of iterations counter}
		\While {\big($(normR/normC)>tol$ \, and \, $\texttt{iter}<maxit$\big)}  
		\State \% First step 
		\State $R=Q_1\texttt{'}*R*Q_2$
		\State $Z=R./(D1+D2.\texttt{'})$  \Comment{ solve $(\alpha I + D_1)Z+Z(\beta I+D_2)=R$}	
		\State $X=X+Q_1*Z*Q_2\texttt{'}$  \Comment{ update $X$; $X\leftarrow X+Q_1ZQ_2^*$}
		\Statex
		\State \% Second step
		\State $R=C-A*X-X*B$
		\State $R=Q_3\texttt{'}*R*Q_4$
		\State $Z=R./(D_3+D_4.\texttt{'})$   \Comment{ solve $(\alpha I+D_3){Z}+{Z}(\beta I+D_4)=R$ }
		
		\State $X=X+Q_3*Z*Q_4\texttt{'}$  \Comment{ update $X$; $X\leftarrow X+Q_3ZQ_4^*$}
		
		\Statex 
		
		\State $R=C-A*X-X*B$
		
		\State $normR=\texttt{norm}(R,\texttt{`fro'})$
		
		\State $\texttt{iter}=\texttt{iter}+1$
		
		\EndWhile	
		\Statex
		
		\If{ $\texttt{iter}>=\texttt{maxit}$}
		\State \texttt{disp}(\texttt{`}Maximum number of iterations exceed.\texttt{'})
		\EndIf
	\end{algorithmic}
\end{algorithm}

\begin{algorithm}
	\caption{\textbf{\; BSSOR --  Block Symmetric Successive Over-Relaxation iteration }}
	\label{alg:algorithm5}
	\begin{algorithmic}
		\Statex
		
		\hspace*{-0.4cm}\textbf{Input:} matrices $A, B$, $C$ (orders $n\times n$, $m\times m$ and $n\times m$),
		
		\hspace*{0.1cm} initial approximation $X_0$, relaxation parameter $\omega$, relative residual tolerance $tol$,
		
		\hspace*{0.1cm}  maximum number of iterations $maxit$ 
		\Statex
		\hspace*{-0.4cm}\textbf{Output:} Solution $X$ of the Sylester equation $AX+XB=C$

\Statex	
\State \% Two steps with the same parameter $\omega$
\State \%  $(D_1/\omega+L_1)X_{k+\frac{1}{2}}+X_{k+\frac{1}{2}}(D_2/\omega+U_2)=C+\big[(1-w)/wD_1-U_1\big]X_k+X_k\big[(1-w)/wD_2-L_2\big]$
\State \%  $(D_1/\omega+U_1)X_{k+1}+X_{k+1}(D_2/\omega+L_2)=C+\big[(1-w)/wD_1-L_1\big]X_{k+\frac{1}{2}}+X_{k+\frac{1}{2}}\big[(1-w)/wD_2-U_2\big]$
	
\Statex
\State \% Using  residuals, $R_k$ and $R_{k+\frac{1}{2}}$, and new variables, $Z_{k+1}$ and $Z_{k+\frac{1}{2}}$
\State \%  $(D_1/\omega+L_1)Z_{k+\frac{1}{2}}+Z_{k+\frac{1}{2}}(D_2/\omega+U_2)=R_k$; \quad $X_{k+\frac{1}{2}}=X_k+Z_{k+\frac{1}{2}}$
\State \%  $(D_1/\omega+U_1)Z_{k+1}+Z_{k+1}(D_2/\omega+L_2)=R_{k+\frac{1}{2}}$; \quad $X_{k+1}=X_{k+\frac{1}{2}}+Z_{k+1}$

\Statex
\State $D_1=\texttt{diag}(A)$; \; $L_1=\texttt{tril}(A,-1)$; \; $U_1=\texttt{triu}(A,1)$ \Comment{ diagonal, strictly lower and }
\State $D_2=\texttt{diag}(B)$; \; $L_2=\texttt{tril}(B,-1)$; \; $U_2=\texttt{triu}(B,1)$ \Comment{ strictly upper parts of $A$ and $B$}

\Statex

\State $D_1=D_1/\omega$; \; $L_1=\texttt{diag}(D_1)+L_1$;\; $U_1=\texttt{diag}(D_1)+U_1$
\State $D_2=D_2/\omega$; \; $L_2=\texttt{diag}(D_2)+L_2$; \; $U_2=\texttt{diag}(D_2)+U_2$

\Statex	
\State  $A=\texttt{sparse}(A);\; B=\texttt{sparse}(B); \; C=\texttt{sparse}(C)$ \Comment{ \textsc{Matlab} sparse matrix storage format}

\State $X=X_0$; \;$R=C-A*X-X*B$ \Comment{ initial residual}
\State $normR=\texttt{norm}(R,\texttt{`fro'})$; \;  
$normC=\texttt{norm}(C,\texttt{`fro'})$ \Comment{ Frobenious norms of $R$ and $C$}
\State $Z$ = \texttt{zeros}($n$,$m$)\Comment{ preallocation for speed}	    
\State $\texttt{index}_1$ = \texttt{eye}($n$,\texttt{`logical'}) \Comment{ logical indexing for the diagonal elements}
\State $\texttt{index}_2$ = \texttt{eye}($m$,\texttt{`logical'}) \Comment{ logical indexing for the diagonal elements}
\State $\texttt{iter}=0$ \Comment{ number of iterations counter}
\While {\big($(normR/normC)>tol$ \, and \, $\texttt{iter}<maxit$\big)}  
\State \% First step                   \Comment{ solve $L_1Z+ZU_2=R$}	
 \State \texttt{opts.LT} = \texttt{true}; \; \texttt{opts.UT} = \texttt{false};   
                          \Comment{ \texttt{LT} - lower triangular option to \texttt{linsolve}}

\State $L_1(\texttt{index}_1)=D_1+D_2(1)$    \Comment{ $L_1=L_1+D_2(1)I_n$}

\State $Z(:,1)=\texttt{linsolve}(L_1,R(:,1),\texttt{opts})$            \Comment{ column $1$ of $Z$}

\For {$k=2:m$}  \Comment{ columns $2$ through $m$ of $Z$}
\State $L_1(\texttt{index}_1)=D_1+D_2(k)$     \Comment{ $L_1=L_1+D_2(k)I_n$}
\State $Z(:,k)=\texttt{linsolve}(L_1,R(:,k)-Z(:,1:(k-1))*U_2(1:(k-1),k),\texttt{opts})$
\EndFor
\State $X=X+Z$; \;  $R=C-A*X-X*B$ \Comment{ update $X$ and $R$}
\State \% Second step \Comment{ solve $U_1Z+ZL_2=R$}
 \State \texttt{opts.LT} = \texttt{false}; \; \texttt{opts.UT} = \texttt{true};   
\Comment{ \texttt{UT} - upper triangular option to \texttt{linsolve}}

\State $L_2(\texttt{index}_2)=D_2+D_1(n)$            \Comment{ $L_2=L_2+D_1(n)I_n$}
\State $Z(n,:)=\texttt{linsolve}(L_2.\texttt{'},R(n,:).\texttt{'},\texttt{opts})$ \Comment{ last row of $Z$}

\For {$k=n-1:-1:1$} \Comment{ rows $n-1$ through $1$ of $Z$}
\State $L_2(\texttt{index}_2)=D_2+D_1(k)$    \Comment{      $L_2=L_2+D_1(k)I_n$  } 
\State $Z(k,:)=\texttt{linsolve}(L_2.\texttt{'},(R(k,:)-U_1(k,k+1:n)*Z(k+1:n,:)).\texttt{'},\texttt{opts})$
\EndFor      

\State $X=X+Z$; \; $R=C-A*X-X*B$  \Comment{ update $X$ and $R$}

\State $normR=\texttt{norm}(R,\texttt{`fro'})$

\State $\texttt{iter}=\texttt{iter}+1$

\EndWhile	

\If{ $\texttt{iter}>=\texttt{maxit}$}
\State \texttt{disp}(\texttt{`}Maximum number of iterations exceed.\texttt{'})
\EndIf
\end{algorithmic}
\end{algorithm}

\begin{algorithm}
	\caption{\textbf{\; \texttt{mylyap} --  Bartels-Stewart method }}
	\label{alg:algorithm6}
	\begin{algorithmic}
		\Statex
		
		\hspace*{-0.4cm}\textbf{Input:} matrices $A, B$, $C$ (orders $n\times n$, $m\times m$ and $n\times m$)

		\Statex
		\hspace*{-0.4cm}\textbf{Output:} Solution $X$ of the Sylester equation $AX+XB=C$
		\Statex
		\State [$Q_1$,\,$T_1$] = \texttt{schur}(\texttt{full}($A\texttt{'}$))     
		\State [$Q_1$,\,$T_1$] =  \texttt{rsf2csf}($Q_1$,$T_1$)              
		\State $T_1=T_1\texttt{'}$  \Comment{ lower complex schur form of $A$; $A=Q_1T_1Q_1^*$}
		
		\Statex
		\State [$Q_2$,\,$T_2$] = \texttt{schur}(\texttt{full}($B$))
		\State [$Q_2$,\,$T_2$] = \texttt{rsf2csf}($Q_2$,$T_2$) \Comment{ upper complex schur form of $B$; $B=Q_2T_1Q_2^*$}
		
		\Statex
		
		\State $dT_1=\texttt{diag}(T_1)$ \Comment{ diagonal elements of $T_1$ and $T_2$ }
		\State $dT_2=\texttt{diag}(T_2)$

		\State \texttt{index} = \texttt{eye}($n$,\texttt{`logical'})
		
		\Statex
		\State \% Solution of $T_1X+XT_2=Q_1^*CQ_2$
		\State $C=Q_1\texttt{'}*C*Q_2$
	    \State $X$ = \texttt{zeros}($n$,$m$)\Comment{ preallocation for spead}	
	    \State \texttt{opts.LT} = \texttt{true}       \Comment{ \texttt{LT} - lower triangular option to \texttt{linsolve}}
	    \State $T_1(\texttt{index})=\texttt{diag}(T_1)+dT_2(1)$        \Comment{ $T_1=T_1+dT_2(1)I_n$}			
		\State $X(:,1)=\texttt{linsolve}(T_1, C(:,1),\texttt{opts})$ \Comment{ column $1$ of $X$}
		
		\Statex
	
	    \For {$k=2:m$}   \Comment{ columns 2 through m of X}
		\State $T_1(\texttt{index})=dT_1+dT_2(k)$ \Comment{ $T_1=T_1+dT_2(k)I_n$}
		\State $X(:,k)=\texttt{linsolve}(T_1,C(:,k)-X(:,1:(k-1))*T_2(1:(k-1),k\big),\texttt{opts})$
		\EndFor
		
		\Statex	    
		\State \% Solution of $AX+XB=C$
		
		\State $X=Q_1*X*Q_2\texttt{'}$
		
	\end{algorithmic}
\end{algorithm}

\end{document}